\documentclass[10pt, twoside]{amsart}
\usepackage{amsmath, amssymb, amsthm}
\usepackage{url}

\newtheorem{thm}{Theorem}[section]
\newtheorem{cor}[thm]{Corollary}
\newtheorem{dfn}[thm]{Definition}
\newtheorem{lem}[thm]{Lemma}
\newtheorem{pro}[thm]{Proposition}
\theoremstyle{definition}
\newtheorem{exa}[thm]{Example}

\numberwithin{equation}{section}

\def\Aut{{\rm Aut}}
\def\CCDS{{\mathcal {CCDS}}}
\def\CSCDS{{\mathcal {CSCDS}}}
\def\dbar{{\overline d}}
\def\Diff{{\rm Diff}}
\def\ev{{\rm ev}}

\def\Hbar{{\overline H}}
\def\hphi{{\widehat \phi}}
\def\Homeo{{\rm Homeo}}
\def\id{{\rm id}}
\def\L{{\mathcal L}}
\def\oneinfty{{(1,\infty)}}
\def\osc{{\rm osc}}
\def\r{{\underline r}}
\def\R{{\mathbb R}}
\def\TCDS{{\mathcal {TCDS}}}
\def\TSCDS{{\mathcal {TSCDS}}}
\def\U{{\mathcal U}}
\def\Z{{\mathbb Z}}
\def\z{{\underline z}}

\title[Topological contact dynamics II]{Topological contact dynamics II: topological automorphisms, contact homeomorphisms, and non-smooth contact dynamical systems}
\author[S.~M\"uller \& P.~Spaeth]{Stefan M\"uller and Peter Spaeth}
\email{mueller@kias.re.kr \text{\it{and}} spaeth@kias.re.kr}
\address{Korea Institute for Advanced Study, Seoul 130-722, Republic of Korea}


\begin{document}
\thispagestyle{plain}

\begin{abstract}
This sequel to our previous paper \cite{ms:tcd11} continues the study of topological contact dynamics and applications to contact dynamics and topological dynamics.
We provide further evidence that the topological automorphism groups of a contact structure and a contact form are the appropriate transformation groups of contact dynamical systems.
The article includes an examination of the groups of time-one maps of topological contact and strictly contact isotopies, and the construction of a bi-invariant metric on the latter.
Moreover, every topological contact or strictly contact dynamical system is arbitrarily close to a continuous contact or strictly contact dynamical system with the same end point.
In particular, the above groups of time-one maps are independent of the choice of norm in the definition of the contact distance.
On every contact manifold we construct topological contact dynamical systems with time-one maps that fail to be Lipschitz continuous, and smooth contact vector fields whose flows are topologically conjugate but not conjugate by a contact $C^1$-diffeomorphism.
\end{abstract}

\maketitle
\section{Introduction} \label{sec:intro}
A contact manifold $(M,\xi)$ comes with two important groups of diffeomorphisms, the group $\Diff (M,\xi)$ of automorphisms preserving the contact structure $\xi$, and the subgroup $\Diff_0 (M,\xi)$ of automorphisms that are time-one maps of isotopies generated by a smooth contact vector field or function.
Similarly, a contact form $\alpha$ on $(M,\xi)$ has attached to it the group of automorphisms $\Diff (M,\alpha) \subset \Diff (M,\xi)$ of diffeomorphisms preserving the contact form, and the subgroup of time-one maps $\Diff_0 (M,\alpha) \subset \Diff (M,\alpha)$ of isotopies generated by a smooth strictly contact vector field or basic function on $(M,\alpha)$.
In the theory of dynamical systems, elements of the above automorphism groups are also known as transformations, and the time-one maps are configurations of dynamical systems, preserving the additional geometric structure.

The study of such automorphism groups, and in particular of the subgroups of time-one maps, associated to a smooth or topological structure and to various geometric structures, such as a volume or symplectic form, a contact structure or contact form, or a good measure, has a long and fruitful history, see for instance the monograph \cite{banyaga:scd97} for a comprehensive account of this area of study.
Several well-known phenomena in contact topology and more recent developments suggest that the above groups of diffeomorphisms are too restrictive.
See below for details.

In this article, we consider the more general groups $\Aut (M,\xi)$ and $\Aut (M,\alpha)$ of topological automorphisms of a contact structure and a contact form, and their subgroups $\Homeo (M,\xi)$ and $\Homeo (M,\alpha)$ of time-one maps of topological contact and strictly contact isotopies that emerge in the context of topological contact dynamics \cite{ms:tcd11}.
Compare to the references \cite{mueller:ghh07, mueller:ghh08, mueller:ghl08, banyaga:uch12} for similar topological automorphism groups in topological Hamiltonian dynamics and topological strictly contact dynamics.

Topological contact dynamics is a natural extension of the dynamics of a smooth contact vector field or function to topological dynamics, and is the odd-dimensional analog to topological Hamiltonian dynamics developed in \cite{mueller:ghh07, mueller:ghh08, mueller:ghl08, viterbo:ugh06, buhovsky:ugh10}, and topological symplectic dynamics considered in \cite{banyaga:gsh08}.
It is also a non-trivial generalization of topological strictly contact dynamics from \cite{banyaga:uch12}.
As in the smooth theories, topological Hamiltonian dynamics of an integral symplectic manifold $(B,\omega)$ is intimately related to topological  strictly contact dynamics of the total space of the associated prequantization bundle with base $B$ \cite{banyaga:uch12}, and in turn, topological contact dynamics of a contact manifold $(M,\xi)$ corresponds to topological Hamiltonian dynamics of the symplectization of $(M,\alpha)$, where $\alpha$ is a contact form with $\ker \alpha = \xi$ \cite{ms:tcd11}.
See the cited references for details.

The above theories have numerous applications to their smooth counterparts, see e.g.\ \cite{ms:tcd11, ms:tcd12} or section~\ref{sec:topo-contact-dyn}, and to topological dynamics in dimensions two and three \cite{ms:hvf11}, by extending a priori smooth invariants to topological Hamiltonian and contact dynamical systems or their time-one maps.
An extensive motivation for the study of topological Hamiltonian dynamics can be found in \cite{mueller:ghh07, mueller:ghh08}, while topological symplectic dynamics and the two flavors of topological contact dynamics mentioned in this introduction are to a large extent motivated by the close connections with their Hamiltonian counterpart.
We recall in particular $C^0$-rigidity of symplectic diffeomorphisms \cite{eliashberg:tsw87, gromov:pdr86, gromov:shs86}, $C^0$-rigidity of strictly contact and contact diffeomorphisms and their conformal factors \cite{ms:tcd11}, and $C^0$-rigidity of Hamiltonian, strictly contact, and contact isotopies and their conformal factors, see \cite{mueller:ghh07, banyaga:uch12, ms:tcd11} or again section~\ref{sec:topo-contact-dyn}.

Additional motivation comes from several implications of the energy-capacity inequality in contact dynamics that is proved in \cite{ms:tcd11}.
It follows that the group of diffeomorphisms preserving a contact form admits a bi-invariant metric \cite{ms:tcd11}.
This generalizes a previous theorem in \cite{banyaga:lci06} establishing the existence of classical diffeomorphism groups other than the group of Hamiltonian diffeomorphisms that support non-degenerate bi-invariant pseudo-metrics.

This article is part of a series of papers concerning topological contact dynamics initiated in \cite{ms:tcd11}.
In section~\ref{sec:contact-geom} we recall the basic notions of contact geometry and smooth contact dynamics.
Sections~\ref{sec:topo-contact-dyn} and \ref{sec:topo-auto} then review topological contact and strictly contact dynamics, and topological automorphisms of a contact structure and a contact form, including the transformation law, a detailed motivation, and some of the applications to contact geometry and smooth contact dynamics that support the point of view that the groups $\Aut (M,\xi)$ and $\Aut (M,\alpha)$ are the correct automorphism groups of a contact structure and a contact form, respectively.

In sections~\ref{sec:contact-homeo} and \ref{sec:contact-topo} we study topological properties of the groups of contact and strictly contact homeomorphisms.
The main focus is on the topology induced by the time-one evaluation map.
When equipped with this contact topology, the groups of time-one maps become first countable topological groups, and therefore admit left invariant metrics.
Moreover, these groups are path-connected and locally path-connected, and any two contact or strictly contact homeomorphisms can be connected by a (short) topological contact or strictly contact isotopy.

The definition of the contact distance on the group of smooth contact dynamical systems involves a choice of norm on the space of tangent vectors to isotopies of contact diffeomorphisms, leading to the notions of topological and continuous contact dynamical systems, and similarly for strictly contact dynamical systems.
Another main focus in this article is on this choice of norm, and centers around the main lemma stated in section~\ref{sec:contact-homeo}.
In brief, we prove that up to an arbitrarily small perturbation with fixed end points, the choice of norm is not relevant.
As a consequence, the groups of time-one maps are independent of the choice of norm.
Sections~\ref{sec:regular} and \ref{sec:rep} set the stage for the proof of the main lemma in section~\ref{sec:proof-main}.
In particular, we show that a generic smooth contact isotopy is regular, i.e.\ the isotopy is never stationary, and investigate how to reparameterize and perturb a given contact isotopy to have certain desirable properties, for instance to traverse its path in the group of contact diffeomorphisms at nearly constant speed.
Moreover, we construct multi-parameter variations of the constant loop that seem to be of independent interest.
Similar results are also proved for strictly contact isotopies.

In section~\ref{sec:energy-capacity} we show how the energy-capacity inequality from \cite{ms:tcd11} leads to a contact energy function and a bi-invariant metric on the group of strictly contact homeomorphisms, with no restrictions on the contact form.
This metric is an analog of the bi-invariant metric on the group of strictly contact diffeomorphisms in \cite{banyaga:lci06, ms:tcd11}.
We moreover demonstrate that the coarse and fine contact energy, and hence the resulting metrics, coincide.

In the final section~\ref{sec:non-smooth}, we produce examples of contact homeomorphisms that fail to be Lipschitz continuous, and in particular are not $C^1$-smooth, and examples of smooth contact vector fields that are topologically conjugate but not conjugate by contact $C^1$-diffeomorphisms.
In particular, for any contact manifold with arbitrary contact form, topological contact dynamics is a natural and genuine extension of the dynamics of a smooth contact vector field or function to topological dynamics.

The last two sections can each be read independently of the rest of the paper immediately after reviewing the relevant definitions in sections~\ref{sec:topo-contact-dyn}, \ref{sec:topo-auto}, and \ref{sec:contact-homeo}.

\section{Review of contact geometry and contact dynamics} \label{sec:contact-geom}
Let $(M,\xi)$ be a smooth manifold of dimension $2 n + 1$, equipped with a \emph{contact structure} $\xi$, that is, a coorientable (or transversally orientable) nowhere integrable (or maximally non-integrable) field of hyperplanes $\xi \subset TM$.
That means the contact structure can be defined as the kernel $\xi = \ker \alpha$ of a one-form $\alpha$ such that $\alpha \wedge (d\alpha)^n$ is a volume form.
Such a one-form is called a \emph{contact form} on $(M,\xi)$.
For simplicity, we assume throughout this article that $M$ is closed, i.e.\ compact and without boundary, and connected.
See the article \cite{ms:tcd11} for the necessary adjustments to be made for general contact manifolds.

A \emph{diffeomorphism} $\phi$ of $M$ is by definition \emph{contact} if it preserves the field of hyperplanes $\xi \subset TM$, i.e.\ $\phi_* \xi = \xi$, or equivalently, $\phi^* \alpha = e^h \alpha$ for a smooth function $h$ on $M$.
A smooth \emph{isotopy} $\Phi = \{ \phi_t \}_{0 \le t \le 1}$ is called \emph{contact} if $\phi_t$ is a contact diffeomorphism for all times $t$.
Denote by $X = \{ X_t \}_{0 \le t \le 1}$ the time-dependent smooth vector field generating the isotopy $\Phi$ in the sense
\begin{align} \label{eqn:inf-gen}
	\frac{d}{dt} \phi_t = X_t \circ \phi_t.
\end{align}
Then $\Phi$ is contact if and only if $\L_{X_t} \alpha = \mu_t \: \alpha$ for a time-dependent smooth function $\mu$ on $M$, where $\L$ denotes the Lie derivative.
In that case, the smooth vector field $X$ is called a \emph{contact vector field}, and the smooth function $H \colon [0,1] \times M \to \R$ defined by $H_t = \alpha (X_t)$ is called its \emph{Hamiltonian}.
The restriction of the two-form $d\alpha$ to the sub-bundle $\xi$ is non-degenerate.
Thus the preceding formula defines a 1--1 correspondence between contact vector fields and smooth functions.
An important contact vector field is the \emph{Reeb vector field} $R = R_\alpha$, the unique vector field in the kernel of $d\alpha$ that is normalized by the identity $\alpha (R_\alpha) = 1$.
More generally, the vector field $X_H = \{ X_H^t \}$ is defined uniquely by the relations
\begin{align} \label{eqn:contact-ham}
	\iota (X_H^t) \alpha = H_t \ \ \mbox{and} \ \ \iota (X_H^t) d\alpha = (R_\alpha . H_t) \alpha - dH_t,
\end{align}
where $R_\alpha . H_t = dH_t (R_\alpha)$.
Then $X_H$ is contact with Hamiltonian $H$ and $\mu_t = R_\alpha . H_t$.
Here $\iota$ denotes interior multiplication by a vector field.
We write $\Phi = \Phi_H$ for the contact isotopy generated by a contact vector field $X = X_H$ with Hamiltonian $H$.
The smooth function $h \colon [0,1] \times M \to \R$ determined by $\phi_t^* \alpha = e^{h_t} \alpha$ is called the \emph{conformal factor} of the isotopy $\Phi = \Phi_H$, and
\begin{align} \label{eqn:conformal-factor}
	h_t = \int_0^t (R_\alpha . H_s) \circ \phi_H^s \, ds.
\end{align}

A triple $(\Phi, H, h)$ is called a smooth \emph{contact dynamical system} if $\Phi = \Phi_H$ is a smooth contact isotopy with Hamiltonian $H$ and conformal factor $h$.
We will usually denote the Hamiltonian by an upper case Roman letter, and the conformal factor by the corresponding lower case letter.
The group of contact diffeomorphism is denoted by $\Diff (M,\xi)$, and $\Diff_0 (M,\xi)$ denotes its identity component.

If $f$ is a smooth function on $M$, then the one-form $e^f \alpha$ defines another contact form on $(M,\xi)$.
We fix a coorientation of $\xi$, and hence an orientation of $M$.
Then every contact form with kernel $\xi$ can be written in this way.
The relation between a contact isotopy and its Hamiltonian and conformal factor, and in particular the definition of the Reeb vector field, depend on this choice of contact form $\alpha$.
However, the notions of contact diffeomorphism and contact vector field depend only on the contact structure $\xi$.
More precisely, if $(\Phi, H, h)$ is a smooth contact dynamical system with respect to the contact form $\alpha$, then $(\Phi, e^f H, h + (f \circ \Phi - f))$ is the corresponding smooth contact dynamical system with respect to the contact form $e^f \alpha$.
Here and in the following, for brevity the expression $h + (f \circ \Phi - f)$ denotes the function whose value at $(t,x) \in [0,1] \times M$ is
	\[ (h + (f \circ \Phi - f)) (t,x) = h_t (x) + (f (\phi_t (x)) - f (x)), \]
and similarly for the Hamiltonian $e^f H$.
For the present purposes, it suffices to fix a choice of contact form $\alpha$ with $\ker \alpha = \xi$ for the remainder of this article.
See the remarks in \cite{ms:tcd11} for details.

A contact \emph{diffeomorphism} $\phi$ is called \emph{strictly contact} if it in fact preserves the contact form $\alpha$, i.e.\ $\phi^* \alpha = \alpha$, or equivalently, it preserves the \emph{canonical volume form} $\nu_\alpha = \alpha \wedge (d\alpha)^n$ induced by $\alpha$.
An \emph{isotopy} $\{ \phi_t \}$ is by definition \emph{strictly contact} if each diffeomorphism $\phi_t$ is strictly contact, or in other words, its conformal factor $h \colon [0,1] \times M \to \R$ vanishes.
Thus a contact isotopy $\{ \phi_t \}$ is strictly contact if and only if its contact vector field $X = X_H$ preserves the contact form $\alpha$, or $\L_{X_t} \alpha = 0$, which in turn is equivalent to $R_\alpha . H_t = 0$, for all times $t$.
The latter means $H_t$ is invariant under the Reeb flow, and such a smooth function $H$ is called \emph{basic}.
The group of strictly contact diffeomorphisms is henceforth denoted by $\Diff (M, \alpha) \subset \Diff (M, \xi)$, its identity component is the group $\Diff_0 (M, \alpha)$, and $(\Phi, H)$ or $(\Phi, H, 0)$ is called a smooth \emph{strictly contact dynamical system} if $\Phi = \Phi_H$ is a smooth strictly contact isotopy with basic Hamiltonian $H$.

For later reference, we recall two well-known lemmas in contact geometry that are used frequently in \cite{ms:tcd11} and in this work.

\begin{lem}[Conformal factors] \label{lem:conformal-factor}
If $\phi$ and $\psi$ are contact diffeomorphisms of $(M,\xi)$ with conformal factors $h$ and $g$, respectively, i.e.\ $\phi^* \alpha = e^h \alpha$ and $\psi^* \alpha = e^g \alpha$ for smooth functions $h$ and $g$ on $M$, then
	\[ (\phi \circ \psi)^* \alpha = e^{h \circ \psi + g} \alpha \ \ \mbox{and} \ \ (\phi^{-1})^* \alpha = e^{- h \circ \phi^{-1}} \alpha. \]
\end{lem}

We often write $H \mapsto \Phi_H$ if $H$ is the Hamiltonian corresponding to the smooth contact isotopy $\Phi_H = \{ \phi_H^t \}$.

\begin{lem}[Hamiltonians] \label{lem:contact-ham}
If $H \mapsto \Phi_H$ and $F \mapsto \Phi_F$, then
\begin{align*}
	& H \# F \mapsto \Phi_H \circ \Phi_F,  & (H \# F)_t = H_t + \left( e^{h_t} \cdot F_t \right) \circ (\phi_H^t)^{-1}, \\
	& \Hbar \mapsto \Phi_H^{-1},  & \Hbar_t = - e^{- h_t} \cdot \left( H_t \circ \phi_H^t \right), \\
	& \Hbar \# F \mapsto \Phi_H^{-1} \circ \Phi_F,  & (\Hbar \# F)_t = e^{- h_t} \cdot \left( (F_t - H_t) \circ \phi_H^t \right), \\
	& K \mapsto \phi^{-1} \circ \Phi_H \circ \phi, & K_t = e^{-g} \left( H_t \circ \phi \right),
\end{align*}
where $\phi \in \Diff (M,\xi)$ with $\phi^*\alpha = e^g \alpha$.
Here composition $\circ$ and inversion are to be understood as composition and inversion of the diffeomorphisms at each time $t$.
\end{lem}

The Lie algebra of the group $\Diff (M,\xi)$ can be identified with the space $C^\infty (M)$ of smooth functions on $M$ via the relation $\alpha (X) = H$.
A natural choice of norm on $C^\infty (M)$ is to define
	\[ \| H \| = \osc (H) + | c (H) | = \left( \max_{x \in M} H (x) - \min_{x \in M} H (x) \right) + \left| \frac{1}{\int_M \nu_\alpha} \cdot \int_M H \, \nu_\alpha \right|, \]
or in other words, the oscillation of the function $H$ plus the absolute value of its mean value with respect to the volume form $\nu_\alpha = \alpha \wedge (d\alpha)^n$ defined above.
For a time-dependent function $H \colon [0,1] \times M \to \R$, we consider two choices of norm, known as the \emph{$L^\oneinfty$-norm} and \emph{$L^\infty$-norm}, given by
\begin{align} \label{eqn:one-infty-norm}
	\| H \|_\oneinfty = \int_0^1 \| H_t \| \, dt
\end{align}
and
\begin{align} \label{eqn:infty-norm}
	\| H \|_\infty = \max_{0 \le t \le 1} \| H_t \|,
\end{align}
respectively.
The two norms $\| \cdot \|_\oneinfty \le \| \cdot \|_\infty$ play a crucial role in this paper.
In order to distinguish the two cases, the prefixes $L^\oneinfty$ and $L^\infty$ will be attached to various objects defined below, and the subscripts or superscripts $\oneinfty$ and $\infty$ will be used where appropriate.

For reasons explained in \cite{ms:tcd11}, we define the norm of the conformal factor $h \colon [0,1] \times M \to \R$ of a contact isotopy by the \emph{maximum norm}
\begin{align} \label{eqn:max-norm}
	| h | = \max_{0 \le t \le 1} \max_{x \in M} | h (t, x) |,
\end{align}
which is in fact equivalent to the norm (\ref{eqn:infty-norm}).

We sometimes also consider contact isotopies that are not necessarily based at the identity.
That is, we consider isotopies $\Phi = \psi \circ \Phi_H = \{ \psi \circ \phi_H^t \}$, where $\Phi_H$ is a contact isotopy in the usual sense, with $\phi_H^0 = \id$, and $\psi \in \Diff (M,\xi)$.
This isotopy solves the ordinary differential equation (\ref{eqn:inf-gen}) with (contact) vector field $\psi_* X_H$ and initial condition $\phi_0 = \psi$, and thus the Hamiltonian of $\Phi$ is the smooth function $F = (e^g \cdot H) \circ \psi^{-1}$, where $\psi^* \alpha = e^g \alpha$.
By a slight abuse of notation, we write $\Phi = \Phi_F$, but explicitly mention the time-zero map $\psi$ in this instance.
Note that one could also work with isotopies of the form $\Phi = \Phi_H \circ \psi$.
The latter solves the same ordinary differential equation as the isotopy $\Phi_H$ with initial condition $\phi_0 = \psi$.
If no explicit mention is made of the time-zero map of an isotopy, it is assumed to be the identity.

\section{Topological contact dynamics} \label{sec:topo-contact-dyn}
A choice of Riemannian metric determines a distance $d$ between points on $M$, and thus a distance between two homeomorphisms $\phi$ and $\psi$ of $M$ by
\begin{align} \label{eqn:riem-dist}
	d (\phi, \psi) = \max_{x \in M} d (\phi (x), \psi (x) ).
\end{align}
This metric induces the compact-open topology, and in particular, the actual choice of Riemannian metric is mostly irrelevant.
The metric in equation (\ref{eqn:riem-dist}) is not complete, but it gives rise to a complete metric $\dbar$ that also induces the compact-open topology on the group of homeomorphisms of $M$, where $\dbar$ is defined by
	\[ \dbar (\phi, \psi) = d (\phi, \psi) + d (\phi^{-1}, \psi^{-1}). \]
In both cases, one also obtains a distance between isotopies $\Phi = \{ \phi_t \}$ and $\Psi = \{ \psi_t \}$, equal to the maximum of the distance of the time-$t$ maps $\phi_t$ and $\psi_t$.
The metric
	\[ \dbar (\Phi, \Psi) = \max_{0 \le t \le 1} \dbar (\phi_t, \psi_t) \]
is again complete, whereas $d (\Phi, \Psi)$ is not complete.
However, if a sequence $\Phi_i$ of isotopies converges uniformly to an isotopy of homeomorphisms $\Phi$, i.e.\ a limit with respect to $d$ exists, then this sequence is also Cauchy with respect to $\dbar$, and converges to the isotopy $\Phi$.
The same remark applies to sequences of homeomorphisms.
The distance $\dbar$ is called the \emph{$C^0$-metric}.
Composition and inversion of diffeomorphisms and of isotopies (at each time $t$) are continuous with respect to the $C^0$-metric.

\begin{dfn}[Contact metric \cite{ms:tcd11}] \label{dfn:contact-metric}
The \emph{contact distance} between smooth contact isotopies $\Phi_H$ and $\Phi_F$ is the number
	\[ d_\alpha (\Phi_H, \Phi_F) = d_\alpha ((\Phi_H, H, h), (\Phi_F, F, f)) = \dbar (\Phi_H, \Phi_F) + | h - f | + \| H - F \|. \]
\end{dfn}

Here $\| \cdot \|$ denotes either one of the two norms (\ref{eqn:one-infty-norm}) or (\ref{eqn:infty-norm}), and the resulting distances are also called the $L^\oneinfty$-contact metric and the $L^\infty$-contact metric.

\begin{dfn}[Topological and continuous contact dynamical system \cite{ms:tcd11}]
Let $\Phi = \{ \phi_t \}$ be a continuous isotopy of homeomorphisms of $M$, and $H$ and $h \colon [0,1] \times M \to \R$ be time-dependent functions on $M$.
The triple $(\Phi, H, h)$ is called a \emph{topological contact dynamical system} if there exists a sequence of smooth contact dynamical systems $(\Phi_{H_i}, H_i, h_i)$, such that the smooth contact isotopies $\Phi_{H_i} = \{ \phi_{H_i}^t \}$ converge uniformly to the isotopy $\Phi$, the corresponding conformal factors converge uniformly to the continuous function $h$, and the sequence $H_i$ of Hamiltonians satisfies $\| H - H_i \|_\oneinfty \to 0$.
The function $H$ is called a \emph{topological Hamiltonian} with \emph{topological contact isotopy} $\Phi$ and \emph{topological conformal factor} $h$.
The set of topological contact dynamical systems is denoted by $\TCDS (M,\alpha)$.
The extension of the $L^\oneinfty$-contact metric $d_\alpha$ to $\TCDS (M,\alpha)$ is again denoted by $d_\alpha$, and the induced topology is called the \emph{contact topology}.

The triple $(\Phi, H, h)$ is called a \emph{continuous contact dynamical system} if the above sequence $H_i$ converges to $H$ with respect to the metric induced by the $L^\infty$-norm (\ref{eqn:infty-norm}).
In that case, the continuous function $H$ is called a \emph{continuous Hamiltonian} with \emph{continuous contact isotopy} $\Phi$ and \emph{continuous conformal factor} $h$.
The subset of continuous contact dynamical systems is denoted by $\CCDS (M,\alpha)$, $d_\alpha$ also denotes the extension of the $L^\infty$-contact metric to $\CCDS (M,\alpha)$, and the induced metric topology is again called the contact topology.

If in the sequences above each $h_i$ is zero, or in other words, the isotopies $\Phi_{H_i}$ are strictly contact and their Hamiltonians are basic, then $(\Phi_H, H)$ is by definition a \emph{topological} or \emph{continuous strictly contact dynamical system}, respectively \cite{banyaga:uch12}.
The resulting sets $\TSCDS (M,\alpha)$ and $\CSCDS (M,\alpha)$ carry induced contact metrics and contact topologies.
Their elements are sometimes also denoted by $(\Phi_H, H, 0)$.
\end{dfn}

We recall that the metric $\dbar$, and the metrics induced by the norms $| \cdot |$ and $\| \cdot \|$, are complete in the following sense: every Cauchy sequence of isotopies with respect to the $C^0$-metric converges uniformly to a continuous isotopy of homeomorphisms, and a Cauchy sequence of time-dependent continuous functions with respect to the metric induced by either of the norms $| \cdot |$ or $\| \cdot \|$ converges to a time-dependent function on $M$, which is continuous in the case of the maximum norm (\ref{eqn:max-norm}) or the $L^\infty$-norm (\ref{eqn:infty-norm}).
In particular, a topological or continuous contact dynamical system is defined uniquely by an equivalence class of Cauchy sequences of smooth contact dynamical systems, and similarly for strictly contact dynamical systems.

In the $L^\oneinfty$-case, the topological Hamiltonian $H \colon [0,1] \times M \to \R$ may not be continuous but only $L^1$ in the variable $t \in [0,1]$.
However, by standard arguments from measure theory, $H_t$ is defined for almost all $t \in [0,1]$, and is a continuous function of the space variable $x \in M$ for each such $t$.
A topological Hamiltonian can be thought of as an element of the space of functions $L^1 ([0,1],C^0 (M))$ of $L^1$-functions of the unit interval taking values in the space $C^0 (M)$ of continuous functions of $M$.
Strictly speaking, such an object is an equivalence class of functions, where two functions are considered equivalent if and only if they agree for almost all $t \in [0, 1]$, but as is customary in measure theory, we will mostly disregard this subtlety in our treatment, and speak of an $L^\oneinfty$-function.
Such functions $H$ can be defined to be any continuous function at the remaining times $t$ belonging to a set of measure zero.

In order to avoid lengthy and repetitive definitions and statements, we restrict attention to topological contact dynamical systems where possible, and consider continuous contact dynamical systems in short remarks after the conclusion of a particular statement or cohesive discussion.

It is shown in \cite{ms:tcd11} that the isotopy $\Phi_H$ and the topological conformal factor $h$ are \emph{uniquely determined} by the function $H$.

\begin{thm}[Uniqueness of topological contact isotopy and topological conformal factor \cite{ms:tcd11}] \label{thm:unique}
If $(\Phi, H, h)$ and $(\Psi, H, g)$ are two topological contact dynamical systems with the same topological Hamiltonian $H$, then $\Phi = \Psi$ and $h = g$.
\end{thm}

This result justifies writing $\Phi = \Phi_H$ for the limit of the smooth isotopies $\Phi_{H_i}$.
As a special case, we have the following rigidity result for smooth contact isotopies and their conformal factors.

\begin{cor}[Rigidity of contact isotopies and their conformal factors \cite{ms:tcd11}]
Suppose $(\Phi_{H_i}, H_i, h_i)$ is a Cauchy sequence of smooth contact dynamical systems, and further suppose that the Hamiltonians $H_i$ converge to a time-dependent \emph{smooth} function $H$ on $M$.
Then $\Phi_{H_i} \to \Phi_H$ and $h_i \to h$ uniformly, where $\Phi_H$ is the smooth contact isotopy generated by the smooth contact vector field $X_H$, and the smooth function $h$ is given by $(\phi_H^t)^* \alpha = e^{h_t} \alpha$.
\end{cor}

In other words, if $\| H - H_i \| \to 0$, and the limit $H$ happens to be a smooth function, then the a priori only continuous limits of the sequences $\Phi_{H_i}$ and $h_i$ must be smooth as well, and coincide with the contact isotopy and conformal factor generated by the limit isotopy $H$.

Moreover, the set of triples $(\Phi_H, H, h)$, where $H$ is a topological Hamiltonian with topological contact isotopy $\Phi_H$ and topological conformal factor $h$, forms a topological group.

\begin{thm}[\cite{ms:tcd11}] \label{thm:topo-gp}
The space $\TCDS (M,\alpha)$ of topological contact dynamical systems with the contact topology forms a topological group under the operation $\circ$ determined by the formula
	\[ (\Phi_H, H, h)^{-1} \circ (\Phi_F, F, f) = (\Phi_H^{-1} \circ \Phi_F, e^{- h} \cdot ((F - H) \circ \Phi_H), f - h \circ \Phi_H^{-1} \circ \Phi_F). \]
The identity element is $(\id, 0, 0)$, and the space of smooth contact dynamical systems with its usual composition forms a topological subgroup.
The subspace $\TSCDS (M,\alpha)$ of topological strictly contact dynamical systems is a topological subgroup, and in turn, the space of smooth strictly contact dynamical systems forms a topological subgroup of $\TSCDS (M,\alpha)$ and $\TCDS (M,\alpha)$.
\end{thm}

We would like to point out that all of the definitions and statements in section~\ref{sec:contact-geom} make sense for time-dependent $C^1$-smooth Hamiltonians for which the continuous vector field $X_H$ defined by the identities~(\ref{eqn:contact-ham}) is uniquely integrable.
Moreover, it is shown in \cite{ms:tcd11} that these generalizations lead to the same definition of a topological contact dynamical system.

\begin{thm}[\cite{ms:tcd11}]
Let $H \colon [0,1] \times M \to \R$ be a continuous function that is continuously differentiable in the variable $x \in M$, and assume the time-dependent continuous vector field $X_H$ is uniquely integrable.
Denote by $\Phi_H$ the continuous isotopy generated by $X_H$, and by $h \colon [0,1] \times M \to \R$ the continuous function defined by equation~(\ref{eqn:conformal-factor}).
Then $(\Phi_H, H, h)$ is a topological contact dynamical system.
\end{thm}

For instance, a $C^2$-smooth Hamiltonian $H$ satisfies these hypotheses.

\begin{lem} \label{lem:c-one-approx-tcds}
Suppose $(\Phi_{H_i}, H_i, h_i)$ is a sequence of topological contact dynamical systems that is Cauchy with respect to the contact metric $d_\alpha$.
Then the sequence $(\Phi_{H_i}, H_i, h_i)$ converges with respect to the contact metric, and the limit $(\Phi_H, H, h)$ is a topological contact dynamical system.
\end{lem}

\begin{proof}
Choose a diagonal subsequence of smooth contact dynamical systems that converges to $(\Phi_H, H, h)$.
\end{proof}

Similar results hold for contact $C^1$-diffeomorphisms, i.e.\ $C^1$-diffeomorphisms that preserve the hyperplane field $\xi \subset TM$, and for topological automorphisms.
In that case the conformal factor is a continuous function.
See the next section for details.
With the exception of section~\ref{sec:non-smooth}, we will therefore in general restrict attention to smooth Hamiltonians and diffeomorphisms that are of class $C^\infty$.

The results in this section are equally valid with topological contact dynamical systems replaced by continuous contact dynamical systems, and continuous strictly contact dynamical systems replacing topological strictly contact dynamical systems in all of the above statements and results.

\section{Topological automorphisms} \label{sec:topo-auto}
A brief motivation of the definitions of a topological automorphism of the contact structure $\xi = \ker \alpha$ and a topological automorphism of the contact form $\alpha$ is in order.
The group $\Diff (M,\xi)$ of contact diffeomorphisms can be viewed as the group of smooth transformations of $M$ that preserve smooth contact dynamical systems.
Recall from section~\ref{sec:contact-geom} that a smooth contact isotopy $\Phi = \Phi_H$ is uniquely determined by the time-dependent smooth function $H \colon [0,1] \times M \to \R$ that is defined by the relation $\alpha (X_H^t) = H_t$.
Here $X_H$ again denotes the smooth contact vector field that generates the isotopy $\Phi$ in the sense of equation~(\ref{eqn:inf-gen}), and which in turn can be obtained from $H$ via (\ref{eqn:contact-ham}).
A contact vector field can be written $X_H = Y_H + H R$, where $Y_H = X_H - H R \in \ker \alpha = \xi$, and $R$ denotes the Reeb vector field of the contact form $\alpha$.
That is, the dynamics of $X_H$ are determined completely by its Reeb components $\alpha (X_H^t) = \alpha (Y_H^t + H_t R) = H_t$.
If $\psi \in \Diff (M,\xi)$ is a contact diffeomorphism with $\psi^* \alpha = e^g \alpha$, then the conjugated isotopy $\psi \circ \Phi_H \circ \psi^{-1} = \{ \psi \circ \phi_H^t \circ \psi^{-1} \}$ is generated by the smooth contact vector field $\psi_* X_H$.
Its Reeb component at time $t$ has the coefficient function
	\[ \alpha (\psi_* X_H^t) = \alpha (\psi_* Y_H^t + (H_t \circ \psi_t^{-1}) \cdot \psi_* R) = (H_t \circ \psi^{-1}) \cdot \alpha (\psi_* R) \]
since $\psi_* \xi = \xi = \ker \alpha$.
Moreover,
	\[ \alpha (\psi_* R) = (\psi^{-1})^* ((\psi^* \alpha) (R)) = (\psi^{-1})^* (e^g \alpha (R)) = e^{g \circ \psi^{-1}}, \]
proving the last identity of Lemma~\ref{lem:contact-ham} with $\psi$ replaced by its inverse.
That means the conjugated smooth contact isotopy $\psi \circ \Phi_H \circ \psi^{-1}$ is determined completely by the Hamiltonian $H$ of the isotopy $\Phi_H$, the \emph{topological} behavior of $\psi$, and the conformal factor $g$ of the contact diffeomorphism $\psi$.
Compare to Theorem~\ref{thm:trafo-law-app} below.

\begin{dfn}[Topological automorphism \cite{ms:tcd11}] \label{dfn:topo-auto}
A homeomorphism $\phi$ of $M$ is a \emph{topological automorphism} of the \emph{contact structure} $\xi$ if there exists a sequence of contact diffeomorphisms $\phi_i \in \Diff (M,\xi)$ that converges uniformly to $\phi$, and the smooth conformal factors $h_i$ given by $\phi_i^* \alpha = e^{h_i} \alpha$ converge uniformly to a function $h$ on $M$.
The continuous function $h$ is called the \emph{topological conformal factor} of the topological automorphism $\phi$.
The homeomorphism $\phi$ is a \emph{topological automorphism} of the \emph{contact form} $\alpha$ if its topological conformal factor $h$ is zero.
The groups of topological automorphisms of the contact structure $\xi$ and of the contact form $\alpha$ are denoted by $\Aut (M,\xi)$ and $\Aut (M,\alpha)$, respectively.
\end{dfn}

It is shown in \cite{ms:tcd11} that the topological conformal factor $h$ is determined uniquely by the homeomorphism $\phi$ and the contact form $\alpha$.
That is, if there exists another sequence of contact diffeomorphisms $\psi_i$ that converges uniformly to the topological automorphism $\phi$, and if the functions $g_i$ given by $\psi_i^* \alpha = e^{g_i} \alpha$ converge uniformly to a function $g$, then we must have $g = h$.
The set $\Aut(M,\xi)$ forms a group, and as the notation suggests, this group does not depend on the choice of contact form $\alpha$.
The topological conformal factor with respect to another contact form $e^f \alpha$ is the continuous function $h + (f \circ \phi - f)$.
Moreover, the formulas for the conformal factors of the composition of two topological automorphisms and of the inverse of a topological automorphism extend those in Lemma~\ref{lem:conformal-factor}.
See \cite{ms:tcd11} for details.

As a consequence of the above, the usual transformation law in smooth contact dynamics extends to topological contact dynamics.
The same statement is valid for continuous contact dynamical systems.

\begin{thm}[Transformation law \cite{ms:tcd11}] \label{thm:trafo-law}
If $(\Phi, H, h)$ is a topological contact dynamical system, and $\psi$ is a topological automorphism of the contact structure $\xi$ with topological conformal factor $g$, then the conjugated system
	\[ \psi^{-1} \circ (\Phi, H, h) \circ \psi = (\psi^{-1} \circ \Phi \circ \psi, e^{-g} (H \circ \psi), h \circ \psi + g - g \circ \psi^{-1} \circ \Phi \circ \psi) \]
is a topological contact dynamical system.
If $(\Phi, H)$ is a topological strictly contact dynamical system, and $\psi$ is a topological automorphism of the contact form $\alpha$, then
	\[ \psi^{-1} \circ (\Phi, H) \circ \psi = (\psi^{-1} \circ \Phi \circ \psi, H \circ \psi) \]
is a topological strictly contact dynamical system.
\end{thm}

As a special case of this theorem, we have the following application to smooth contact dynamics concerning topologically conjugate smooth contact vector fields.
The proof uses the Uniqueness Theorem~\ref{thm:unique}.

\begin{thm}[\cite{ms:tcd11}] \label{thm:trafo-law-app}
Suppose $\{ \phi_H^t \}$ and $\{ \phi_F^t \}$ are smooth contact isotopies, and $\phi$ is a topological automorphism of the contact structure $\xi$ with topological conformal factor $g$.
If $H = e^{- g} (F \circ \phi)$, then $\{ \phi_H^t \} = \{ \phi^{-1} \circ \phi_F^t \circ \phi \}$.
\end{thm}

\begin{cor}[\cite{ms:tcd11}] \label{cor:reeb}
Suppose $\phi$ is a topological automorphism of $\xi = \ker \alpha$ with topological conformal factor $g$, and assume that $g$ is smooth.
Then the Reeb vector fields of the contact forms $\alpha$ and $e^g \alpha$ are topologically conjugate.
\end{cor}

The converses to Theorems~\ref{thm:unique} and \ref{thm:trafo-law-app} will be proved in the sequel \cite{ms:tcd12}.
See section~\ref{sec:non-smooth} for examples of topologically conjugate smooth contact vector fields that are not conjugate by contact $C^1$-diffeomorphisms, and \cite{ms:hvf11} for the case of topologically conjugate smooth strictly contact and Hamiltonian vector fields.

The proof of Theorem~\ref{thm:trafo-law} consists of showing that if $(\Phi_{H_i}, H_i, h_i)$ is a sequence of smooth contact dynamical systems that converges with respect to the contact metric $d_\alpha$ to the topological contact dynamical system $(\Phi, H, h)$, then the smooth contact dynamical systems $\psi^{-1} \circ (\Phi_{H_i}, H_i, h_i) \circ \psi$ converge to $\psi^{-1} \circ (\Phi, H, h) \circ \psi$.
In the special case $H_i = 1$ for all $i$ (i.e.\ $\Phi_{H_i}$ is the Reeb flow), the Hamiltonians $e^{-g_i} \cdot 1$ must converge uniformly.
Thus the assumption of uniform convergence of the conformal factors in Definition~\ref{dfn:topo-auto} is not only sufficient but also \emph{necessary} to prove this form of convergence in the extension of the transformation law.

For further motivation of Definition~\ref{dfn:topo-auto}, recall that if $\phi$ is a diffeomorphism of the contact manifold $(M, \xi)$, and $h$ is a smooth function on $M$, then the lifted diffeomorphism $\hphi (x,\theta) = (\phi (x), \theta - h (x))$ of the symplectization $(M \times \R, - d (e^\theta \alpha))$ of $(M,\alpha)$ is symplectic, if and only if $\phi$ is contact with $\phi^* \alpha = e^h \alpha$.
By definition, a symplectic homeomorphism is the $C^0$-limit of symplectic diffeomorphisms \cite{mueller:ghh07}.
Then $\phi$ is a topological automorphism of the contact structure $\xi$ with topological conformal factor $h$, if and only if the homeomorphism $\hphi (x,\theta) = (\phi (x), \theta - h (x))$ is an admissible symplectic homeomorphism of the symplectization of $(M,\alpha)$ \cite{ms:tcd11}.
Topological automorphisms of a contact structure are $C^0$-rigid in the following sense.

\begin{thm}[$C^0$-rigidity of contact diffeomorphisms and their conformal factors \cite{ms:tcd11}] \label{thm:rigidity}
Let $\phi \in \Aut (M,\xi)$ be a topological automorphism of $\xi = \ker \alpha$ with topological conformal factor $h$.
If $\phi$ is smooth, then the function $h$ is smooth, and $\phi$ is a contact diffeomorphism with $\phi^* \alpha = e^h \alpha$.
\end{thm}

In fact, the statement of the theorem is a local statement.
That is, if $\phi$ is smooth at a point, then in a neighborhood of that point, $\phi$ is a local diffeomorphism, $h$ is smooth, and $\phi^* \alpha = e^h \alpha$.
Note that a priori $h$ need not necessarily be a smooth function, so that the lift $\hphi (x,\theta) = (\phi (x), \theta - h (x))$ to the symplectization is a priori not a smooth map.

The following two lemmas help put Definition~\ref{dfn:topo-auto} into a sharper perspective.

\begin{lem}
Suppose $\phi_i$ is a sequence of contact diffeomorphisms with $\phi_i^* \alpha = e^{h_i} \alpha$ that converges uniformly to a homeomorphism $\phi$.
Then for every open subset $U$ of $M$, the average values
	\[ \int_U e^{(n + 1) h_i} \nu_\alpha \longrightarrow \int_U (\phi^{-1})_* \nu_\alpha > 0 \ \ \mbox{as} \ \ i \to \infty. \]
\end{lem}

\begin{proof}
The push-forward measures $(\phi_i^{-1})_* \nu_\alpha$ are given by integration of the volume forms $\phi_i^* \nu_\alpha = e^{(n + 1) h_i} \nu_\alpha$.
Since $\phi$ is a homeomorphism, the sequence $\phi_i$ converges with respect to the $C^0$-metric, and thus the induced measures converge in the weak metric to the measure $(\phi^{-1})_* \nu_\alpha$ \cite{fathi:sgh80}.
Evaluation on an open subset is lower semi-continuous, and evaluation on a closed subset is upper semi-continuous \cite{denker:etc76}.
The stated convergence then follows from the fact that integration over $U$ and its closure $\overline{U}$ coincide.
\end{proof}

However, $C^0$-convergence of a sequence of contact diffeomorphisms alone is not sufficient to guarantee even point-wise convergence of the functions $h_i$.

\begin{lem}
Suppose $\phi_i$ is a sequence of contact diffeomorphisms of $M$, $x \in M$, and $c \in [- \infty, + \infty]$.
Then there exists another sequence of contact diffeomorphisms $\psi_i$ with $\psi_i^* \alpha = e^{g_i} \alpha$, so that $g_i (x) \to c$ as $i \to \infty$, and for every open neighborhood $U$ of $x$, we have $\psi_i |_U = \phi_i |_U$ for $i \ge i_U$ sufficiently large.
In particular, the sequence $\phi_i^{-1} \circ \psi_i$ converges to the identity uniformly.
\end{lem}

\begin{proof}
The proof follows easily from Darboux's theorem.
Let $f_i$ be any smooth functions on $M$ with $f_i (x) \to c$ as $i \to \infty$, and $U_1 \supset U_2 \supset U_3 \supset \ldots$ be a nested neighborhood basis of the compact-open topology of $M$ at the point $x \in M$.
By Darboux's theorem, there exist diffeomorphisms $\varphi_i$ that are compactly supported in $U_i$, and interchange the contact forms $\phi_i^* \alpha = e^{h_i} \alpha$ and $e^{f_i} \alpha$ near the point $x$.
Then the sequence $\psi_i = \phi_i \circ \varphi_i$ has the desired properties.
\end{proof}

In order to gain a better understanding of the conformal factor $h$ and of the hypotheses of Definition~\ref{dfn:topo-auto}, recall again that if $\phi$ is a contact diffeomorphism of the contact structure $\xi = \ker \alpha$ with conformal factor $h$, then
	\[ e^h = e^h \alpha \, (R) = (\phi^* \alpha) (R) = \phi^* (\alpha (\phi_* R)) = \alpha (d\phi (R)). \]
That is, the function $e^h$ is the partial derivative of $\phi$ in the Reeb direction and along the Reeb orbits, or the infinitesimal translation by $\phi$ of the hyperplanes $\xi \subset TM$ along the Reeb orbits of the contact form $\alpha$.
As remarked above, given a sequence of contact diffeomorphisms, convergence of these partial derivatives is independent of the choice of contact form $\alpha$.
Let us picture this in local Darboux coordinates.
By the contact neighborhood theorem, near a point $p \in M$ we may choose local coordinates $z, u_1, \ldots, u_{2 n}$, so that a tubular neighborhood of the piece of Reeb orbit through $p$ is parameterized by pieces of Reeb orbits, where the variable $z$ parameterizes the piece of Reeb orbit through $p$, $R = \frac{\partial}{\partial z}$, and $\frac{\partial}{\partial u_i} \in \xi$ is a basis of the normal bundle to the Reeb orbit near $p$.
In these local coordinates,
	\[ e^h = \alpha (d\phi (R)) = \alpha \left( \sum_{i = 1}^{2 n} \frac{\partial (\z \circ \phi)}{\partial u_i} \frac{\partial}{\partial u_i} + \frac{\partial (\z \circ \phi)}{\partial z} \frac{\partial}{\partial z} \right) = \frac{\partial (\z \circ \phi)}{\partial z}, \]
where $\z$ denotes the projection to the zero section of the normal bundle of the piece of Reeb orbit through $p$.
Given a sequence of contact diffeomorphisms $\phi_i$, by the chain rule convergence of the partial derivatives $e^{h_i}$ does not depend on the choice of Darboux coordinates.

\begin{lem} \label{lem:c-one-approx-diff}
Suppose $\phi$ is a contact $C^1$-diffeomorphism, and that the second-order partial derivatives $X . (\alpha (d\phi (R)))$ exist and are continuous for all $X \in \xi$.
Then $\phi$ can be $C^0$-approximated by contact diffeomorphisms $\phi_i$, and if $\phi_i^* \alpha = e^{h_i} \alpha$ and $\phi^* \alpha = e^h \alpha$, then the smooth functions $h_i$ converge to the continuous function $h$ uniformly.
In particular, $\phi$ is a topological automorphism of the contact structure $\xi = \ker \alpha$ with topological conformal factor $h$.
\end{lem}

For example, a contact $C^2$-diffeomorphism satisfies the hypotheses of the lemma.
In the above local coordinates $z, u_1, \ldots, u_{2 n}$, the condition on the second-order partial derivatives in the lemma is that the partial derivatives
	\[ \frac{\partial}{\partial u_i} (\alpha (d\phi (R))) = \frac{\partial}{\partial u_i} e^h = \frac{\partial^2 (\z \circ \phi)}{\partial u_i \, \partial z} \]
exist and are continuous for $i = 1, \ldots, 2 n$.
As the proof given below shows, it is equivalent to assume the second-order partial derivatives $R . (\alpha (d\phi (X)))$ exist and are continuous for all $X \in \xi$.

Before giving the proof of Lemma~\ref{lem:c-one-approx-diff}, we first prove another lemma that together with Lemma~\ref{lem:c-one-approx-diff} gives precise meaning to the statement after Lemma~\ref{lem:c-one-approx-tcds}, regarding contact $C^1$-diffeomorphisms and topological automorphisms of a contact structure or a contact form.

\begin{lem}
Let $\phi_i$ be a sequence of topological automorphisms of the contact structure $\xi$ with topological conformal factors $h_i$, and suppose the homeomorphisms $\phi_i$ converge uniformly to a homeomorphism $\phi$, and the continuous functions $h_i$ converge uniformly to a function $h$.
Then $\phi$ is a topological automorphism of the contact structure $\xi$ with topological conformal factor $h$.
\end{lem}

\begin{proof}
Choose a diagonal subsequence of contact diffeomorphisms that converge uniformly together with their conformal factors.
\end{proof}

Analogous results for strictly contact $C^1$-diffeomorphisms and for topological automorphisms of the contact form $\alpha$ follow immediately, and the corresponding results for symplectic $C^1$-diffeomorphisms can be proved similarly.

\begin{proof}[Proof of Lemma~\ref{lem:c-one-approx-diff}]
Consider the set of $C^1$-diffeomorphisms $\varphi$ of $M$ for which the second-order partial derivatives $X . (\alpha (d\varphi (R)))$ exist and are continuous for all $X \in \xi$.
Equip this set with the topology induced by the subbasis consisting of sets of the form ${\mathcal N} (\varphi; U, V \to \R^{2 n + 1}; K; \epsilon)$, where $U, V \to \R^{2 n + 1}$ are local Darboux coordinates as above, $K \subset U$ is a compact subset with $\varphi (K) \subset V$, and $0 < \epsilon \le + \infty$.
Then $\psi \in {\mathcal N} (\varphi; U, V \to \R^{2 n + 1}; K; \epsilon)$ if $\psi (K) \subset V$ and the local representations of the restrictions to $K$ of $\varphi$ and $\psi$ together with their first-order partial derivatives and the second-order partial derivatives specified above are within $\epsilon$ of each other.
By a straightforward modification of a classical theorem in differential topology (see for instance Theorem~2.7 in \cite{hirsch:dt94}), there exist $C^\infty$-diffeomorphisms $\psi_i$ that $C^1$-converge to $\phi$, and moreover, the partial derivatives $X . (\alpha (d\psi_i (R)))$ converge uniformly to the continuous function $X . (\alpha (d\phi (R)))$ for all $X \in \xi$.

Define smooth functions $h_i$ by $(\psi_i^* \alpha) (R) = e^{h_i}$, and write $\psi_i^* \alpha = e^{h_i} \alpha + \beta_i$ for unique one-forms $\beta_i$ with $\beta_i (R) = 0$.
By hypothesis, the one-forms $\psi_i^* \alpha$ converge uniformly to the one form $\phi^* \alpha = e^h \alpha$, so the functions $h_i$ converge to $h$ and the one-forms $\beta_i$ converge to zero uniformly.
Moreover,
	\[ d (e^{h_i}) \wedge \alpha + e^{h_i} d \alpha + d\beta_i = d (\psi_i^* \alpha) = \psi_i^* (d\alpha) \longrightarrow \phi^* (d\alpha) = d (\phi^* \alpha) = d (e^h \alpha). \]
The function $e^h$ need not be $C^1$-smooth, but the one-form $e^h \alpha$ is continuously differentiable.
By the hypothesis on the second-order partial derivatives, the two-forms $d (e^{h_i}) \wedge \alpha$ converge uniformly to the two-form $d (e^h) \wedge \alpha$, which makes sense because the latter only contains partial derivatives of $e^h$ in the directions of the hyperplanes $\xi$.
Therefore the two-forms $d\beta_i$ converge to zero uniformly.

Define a sequence of one-parameter families of one-forms by
	\[ \alpha_i^t = (1 - t) e^{h_i} \alpha + t \, \psi_i^* \alpha = e^{h_i} \alpha + t \, \beta_i \]
with $d\alpha_i^t = d (e^{h_i}) \wedge \alpha + e^{h_i} d\alpha + t \, d\beta_i$.
Since $\beta_i$ and $d\beta_i$ converge to zero uniformly, the $(2 n + 1)$-forms $\nu_{\alpha_i^t} > 0$ for $i$ sufficiently large, and $\alpha_i^t$ is a one-parameter family of contact forms.
By Gray's stability theorem (see e.g.\ \cite{geiges:ict08}), there exist families of vector fields $X_i^t$ so that $(\phi_{X_i}^t)^* \alpha_i^t = e^{f_i^t} \alpha_i^0 = e^{h_i + f_i^t} \alpha$ for all $t$ and all $i$ sufficiently large, where
	\[ f_i^t = \int_0^t g_i^s \circ \phi_{X_i}^s \, ds, \]
and the functions $g_i^t$ are defined below.
Define $\phi_i = \psi_i \circ \phi_{X_i}^1$.
The diffeomorphisms $\phi_i$ are contact with $\phi_i^* \alpha = e^{h_i + f_i^1} \alpha$.

Since the one-forms $\alpha_i^t$ are smooth, the vector fields $X_i^t$ and their flows are $C^\infty$-smooth, and are given by the defining relations $X_i^t \in \ker \alpha_i^t$ and
	\[ \iota (X_i^t) d\alpha_i^t = g_i^t \alpha_i^t - \frac{d}{dt} \alpha_i^t = g_i^t \alpha_i^t - \beta_i, \]
where the smooth functions $g_i^t$ are defined by
	\[ g_i^t = \left( \frac{d}{dt} \alpha_i^t \right) (R_i^t) = \beta_i (R_i^t), \]
and $R_i^t$ denotes the Reeb vector field of $\alpha_i^t$.
In particular, if $\psi_i$ is contact at a point $p$, then $X_i^t (p) = 0$, and thus the isotopy $\{ \phi_{X_i}^t \}$ is stationary at $p$, and $\phi_i (p) = \psi_i (p)$.

Since the one-forms $\alpha_i^t$ and the two-forms $d\alpha_i^t$ converge uniformly to $\phi^* \alpha$ and $\phi^* (d\alpha)$, respectively, the smooth vector fields $R_i^t$ converge uniformly to the vector field $\phi_* R$, and in particular have bounded coefficients independent of $i$ and $t$.
Since $\beta_i \to 0$ uniformly, the functions $g_i^t$ converge to zero uniformly.
As a consequence, the vector fields $X_i^t$ converge to zero uniformly as well.
By the standard continuity theorem from the theory of ordinary differential equations, their flows $\phi_{X_i}^t$ converge to the identity uniformly.
Therefore the sequence $\phi_i$ converges to $\phi$ uniformly, and the conformal factors $f_i$ converge to zero uniformly.
\end{proof}

Similarly to the case of smooth isotopies, we can also consider topological or continuous contact isotopies whose time-zero map is not necessarily the identity.
That is, we consider isotopies $\Phi = \psi \circ \Phi_H = \{ \psi \circ \phi_H^t \}$, where $\Phi_H$ is a topological contact isotopy in the usual sense, with $\phi_H^0 = \id$, and $\psi \in \Aut (M,\xi)$.
This isotopy corresponds to the topological Hamiltonian $F = (e^g \cdot H) \circ \psi^{-1}$, where $g$ is the unique topological conformal factor of the homeomorphism $\psi$.
Again by a slight abuse of notation, we write $\Phi_F = \Phi$.
Note that one could also work with isotopies of the form $\Phi = \Phi_H \circ \psi$.
Again if no explicit mention of the time-zero map of an isotopy is made, it is assumed to be the identity.

\section{Contact homeomorphisms} \label{sec:contact-homeo}
Denote by
\begin{align} \label{eqn:ev-homeo}
	\ev_1 \colon \TCDS (M,\alpha) \rightarrow \Homeo (M), \ (\Phi_H, H, h) \mapsto \phi_H^1,
\end{align}
the \emph{time-one evaluation map} that assign to a topological contact dynamical system $(\Phi_H, H, h)$ the time-one map $\phi = \phi_H^1$ of the isotopy $\Phi_H$.

\begin{dfn}[Contact homeomorphism \cite{ms:tcd11}]
A \emph{contact homeomorphism} is the time-one map of a topological contact isotopy.
The group $\Homeo (M,\xi)$ of contact homeomorphisms is the image of the time-one evaluation map (\ref{eqn:ev-homeo}).
\end{dfn}

It is shown in \cite{ms:tcd11} that $\Homeo (M,\xi)$ indeed forms a group, and thus is a topological subgroup of the group $\Homeo (M)$ of homeomorphisms of the manifold $M$ with the $C^0$-topology induced by the $C^0$-metric.
Moreover,
\begin{align} \label{eqn:set-inclusions}
	\Diff (M,\xi) \subset \Homeo (M,\xi) \vartriangleleft \Aut (M,\xi) \subset \Homeo (M),
\end{align}
where the second inclusion is as a normal subgroup in the group of topological automorphisms of the contact structure $\xi$.
Properness of the first inclusion is proved in section~\ref{sec:non-smooth} for all contact manifolds.

\begin{pro}[\cite{ms:tcd11}] \label{pro:path-connect}
The group $\Homeo (M,\xi)$ is path-connected under the $C^0$-topology.
Thus $\Homeo (M,\xi) \vartriangleleft \Aut_0 (M,\xi) \subset \Homeo_0 (M)$.
\end{pro}

In order to give the proof, we need to recall the following lemma from \cite{ms:tcd11}.

\begin{lem}[\cite{ms:tcd11}] \label{lem:rep}
Let $(\Phi_H, H, h)$ be a topological contact dynamical system.
For every $s \in [0,1]$, the triple $(\Phi_{H^s}, H^s, h^s)$ is also a topological contact dynamical system with time-one map $\phi_H^s$, where $\Phi_{H^s} = \{ \phi_{H^s}^t \} = \{ \phi_H^{s t} \}$, and the topological Hamiltonian $H^s$ and the topological conformal factor $h^s$ are given by the formulas $H^s (t, x) = s H (s t, x)$ and $h^s (t, x) = h (s t, x)$.
\end{lem}

\begin{proof}[Proof of Proposition~\ref{pro:path-connect}]
The first statement follows at once from Lemma~\ref{lem:rep}, and the inclusions in the second statement are consequences of (\ref{eqn:set-inclusions}).
\end{proof}

A different topology that takes into account the topological \emph{and} the dynamical nature of contact homeomorphisms defined as time-one maps of topological contact isotopies is studied in the next section.

The definitions and results in this section are again also valid with topological contact dynamical systems replaced by continuous contact dynamical systems.
In order to distinguish the two cases, we attach the subscripts or superscripts $\oneinfty$ and $\infty$ where appropriate.
In particular, the groups of contact homeomorphisms $\Homeo_\oneinfty (M,\xi)$ and $\Homeo_\infty (M,\xi)$ denote the time-one maps of topological and continuous contact dynamical systems, respectively.
In this case, the distinction is actually not necessary.

\begin{thm} \label{thm:oneinfty-infty}
The two groups $\Homeo_\oneinfty (M,\xi)$ and $\Homeo_\infty (M,\xi)$ of contact homeomorphisms coincide.
\end{thm}

In view of this theorem, we may omit the subscripts from the notation.
This result is the analog of a theorem concerning the Hamiltonian homeomorphism group of a symplectic manifold that appeared in \cite{mueller:ghh08, mueller:ghl08}, and the line of proof follows the one given there.
The two main differences in the present case are the additional term $| c ( H_t ) |$ in the norm $\| H_t \|$ at each time $t$, and the appearance of the conformal factors in the formulas for composition and inversion of contact isotopies as well as for conjugation by a contact diffeomorphism.
We will in fact demonstrate the following more technical result.

\begin{lem}[Main Lemma] \label{lem:main}
Let $(\Phi_H, H, h)$ be a topological contact dynamical system.
Then there exists a continuous contact dynamical system $(\Phi_F, F, f)$ with the same time-one map $\phi_F^1 = \phi_H^1$.
Given $\epsilon > 0$, the continuous contact dynamical system $(\Phi_F, F, f)$ can be chosen so that either
\begin{align} \label{eqn:small-perturbation}
	\dbar (\Phi_F, \Phi_H) < \epsilon, \ \ | f - h | < \epsilon, \ \ \mbox{and} \ \ \| F - H \|_\oneinfty < \epsilon,
\end{align}
or
\begin{align} \label{eqn:close-to-id}
	\dbar (\Phi_F, \id) < \dbar (\Phi_H, \id) + \epsilon, \ \ | f | < | h | + \epsilon, \ \ \mbox{and} \ \ \| F \|_\infty < \| H \|_\oneinfty + \epsilon.
\end{align}
In fact, $(\Phi_F, F, f)$ is smooth everywhere except possibly at time one, i.e.\ the maps $(t, x) \mapsto \phi_F^t (x)$, $(t, x) \mapsto (\phi_F^t (x))^{-1}$, $(t, x) \mapsto F_t (x)$, and $(t, x) \mapsto f_t (x)$ are smooth except possibly at $t = 1$.
\end{lem}

In section~\ref{sec:regular} we explain a procedure for regularizing a smooth contact isotopy that is similar to the regularization of smooth Hamiltonian isotopies carried out in Section~5.2 in \cite{polterovich:ggs01}.
This step is crucial in the proof in the Hamiltonian case in \cite{mueller:ghl08, mueller:ghh08} and in the contact case in the present paper.
It implies that generically, in a sense to be made precise below, the tangent vector to a contact isotopy never vanishes.
As in the case of curves in finite-dimensional manifolds, one can then reparameterize the isotopy to have nearly constant speed throughout.
After some preparations in the subsequent section~\ref{sec:rep}, Lemma~\ref{lem:main} will be proved in Section~\ref{sec:proof-main}.
Assuming the Main Lemma~\ref{lem:main}, we first deduce Theorem~\ref{thm:oneinfty-infty}.

\begin{proof}[Proof of Theorem~\ref{thm:oneinfty-infty}]
The inclusion $\Homeo_\infty (M,\xi) \subset \Homeo_\oneinfty (M,\xi)$ follows immediately from the definitions.
To prove the theorem, it only remains to show $\Homeo_\oneinfty (M,\xi) \subset \Homeo_\infty (M,\xi)$.
Let $\phi \in \Homeo_\oneinfty (M,\xi)$.
By definition, there exists a topological contact dynamical system $(\Phi_H, H, h)$ such that $\phi = \phi_H^1$.
By Lemma~\ref{lem:main}, there exists a continuous contact dynamical system $(\Phi_F, F, f)$ with the same time-one map $\phi_F^1 = \phi$, and thus $\phi \in \Homeo_\infty (M,\xi)$.
\end{proof}

Denote by
\begin{align} \label{eqn:ev-strictly-homeo}
	\ev_1 \colon \TSCDS (M,\alpha) \rightarrow \Homeo (M), \ (\Phi_H, H) \mapsto \phi_H^1
\end{align}
the \emph{time-one evaluation map} that maps a topological strictly contact dynamical system $(\Phi_H, H)$ to its time-one map $\phi_H^1$.
This notation is not ambiguous, since the restriction of the time-one evaluation map (\ref{eqn:ev-homeo}) on $\TCDS (M,\alpha)$ to $\TSCDS (M,\alpha)$ coincides with the time-one evaluation map (\ref{eqn:ev-strictly-homeo}).
The same observation applies to the restrictions to continuous contact and strictly contact dynamical systems.

\begin{dfn}[Strictly contact homeomorphism \cite{banyaga:uch12}]
The time-one map of a topological strictly contact isotopy is a \emph{strictly contact homeomorphism}.
The group $\Homeo (M,\alpha)$ of strictly contact homeomorphisms is the image of the map (\ref{eqn:ev-strictly-homeo}).
\end{dfn}

By \cite{banyaga:uch12}, the set $\Homeo (M,\alpha)$ indeed forms a group, and thus is a topological subgroup of $\Homeo (M)$ with the $C^0$-topology.
Moreover \cite{banyaga:uch12, ms:tcd11},
\begin{align} \label{eqn:strictly-set-inclusions}
	\Diff (M,\alpha) \subset \Homeo (M,\alpha) \vartriangleleft \Aut (M,\alpha) \subset \Homeo (M),
\end{align}
and the first inclusion is proper if the contact form $\alpha$ is {\em regular} \cite{banyaga:uch12} .
Concerning strictly contact homeomorphisms, we have the following results.

\begin{pro}[\cite{banyaga:uch12, ms:tcd11}]
The strictly contact homeomorphism group is path-connected in the $C^0$-topology, and $\Homeo (M,\alpha) \vartriangleleft \Aut_0 (M,\alpha) \subset \Homeo_0 (M)$.
\end{pro}

All of the above definitions make sense for continuous strictly contact dynamical systems, and the preceding results and their proofs are verbatim the same.

\begin{thm} \label{thm:strictly-infty-oneinfty}
The strictly contact homeomorphism groups $\Homeo_\oneinfty (M,\alpha)$ and $\Homeo_\infty (M,\alpha)$ coincide.
\end{thm}

This theorem appears in \cite{banyaga:uch12} for regular contact manifolds.
We will prove the theorem by establishing a result similar to the Main Lemma~\ref{lem:main} for topological strictly contact dynamical systems.
See section~\ref{sec:proof-main} for the proof.

\begin{lem} \label{lem:strictly-main}
Let $(\Phi_H, H)$ be a topological strictly contact dynamical system.
Then there exists a continuous strictly contact dynamical system $(\Phi_F, F)$ with the same time-one map $\phi_F^1 = \phi_H^1$.
Given $\epsilon > 0$, the continuous strictly contact dynamical system $(\Phi_F, F)$ can be chosen so that either
	\[ \dbar (\Phi_F, \Phi_H) < \epsilon \ \ \mbox{and} \ \ \| F - H \|_\oneinfty < \epsilon, \]
or
	\[ \dbar (\Phi_F, \id) < \dbar (\Phi_H, \id) + \epsilon \ \ \mbox{and} \ \ \| F \|_\infty < \| H \|_\oneinfty + \epsilon. \]
In fact, $(\Phi_F, F)$ is smooth everywhere except possibly at time $t = 1$.
\end{lem}

\begin{proof}[Proof of Theorem~\ref{thm:strictly-infty-oneinfty}]
The proof is verbatim the same as the proof of Theorem~\ref{thm:oneinfty-infty} with the reference to Lemma~\ref{lem:main} replaced by a citation of Lemma~\ref{lem:strictly-main}.
\end{proof}

\section{The contact topology} \label{sec:contact-topo}
Recall that the contact topology on the space $\TCDS (M,\alpha)$ of topological contact dynamical systems is the metric topology induced by the contact metric
	\[ d_\alpha ( (\Phi_H, H, h), (\Phi_F, F, f) ) = \dbar (\Phi_H, \Phi_F) + | h - f | + \| H - F \| \]
defined in section~\ref{sec:topo-contact-dyn}.
The time-one evaluation map
\begin{align} \label{eqn:ev}
	\ev_1 \colon \TCDS (M,\alpha) \rightarrow \Homeo (M,\xi), \ (\Phi_H, H, h) \mapsto \phi_H^1
\end{align}
assigns to a topological contact dynamical system $(\Phi_H, H, h)$ the time-one map of the isotopy $\Phi_H$.
This map $\ev_1$ is by definition surjective, and thus induces the usual quotient topology on the set $\Homeo (M,\xi)$, called the \emph{contact topology} on $\Homeo (M,\xi)$.
By definition, the evaluation map (\ref{eqn:ev}) is continuous.
The metric $d_\alpha$ does not necessarily project to a metric on $\Homeo (M,\xi)$, since $d_\alpha$ is neither left nor right invariant, and thus the triangle inequality may be violated, cf.\ \cite{mueller:ghh08}.
See \cite{ms:tcd11} for remarks on the failure of left and right invariance of the contact metric on $\TCDS (M,\alpha)$.

On the other hand, the usual composition of homeomorphisms induces a group structure on the set $\Homeo (M,\xi) \subset \Homeo (M)$, and the time-one evaluation map $\ev_1$ becomes a homomorphism.
Therefore the contact topology on $\Homeo (M,\xi)$ is metrizable.
In fact, both topological spaces $\TCDS (M,\alpha)$ and $\Homeo (M,\xi)$ admit left invariant metrics.
This follows from the next two results.

\begin{thm}[\cite{birkhoff:tg36, kakutani:mtg36, klee:img52}]
A topological group admits a left invariant metric if and only if it is first countable.
\end{thm}

Note that in our notation a topological group is assumed to be Hausdorff.
See the references in the theorem for explicit constructions of such left invariant metrics.

\begin{thm} \label{thm:ev-topo-gp}
The projection map $ev_1$ induces the structure of a first countable topological group on the topological space $\Homeo (M,\xi)$ with the contact topology.
\end{thm}

\begin{cor}
The space $\Homeo (M,\xi)$ equipped with the contact topology admits a left invariant metric that generates its topology.
\end{cor}

\begin{proof}[Proof of Theorem~\ref{thm:ev-topo-gp}]
By Theorem~\ref{thm:topo-gp}, the space $\TCDS (M,\alpha)$ forms a topological group with the contact topology induced by the contact metric $d_\alpha$.
In particular, the contact topology on $\TCDS (M,\alpha)$ is first countable, and left and right translations in $\TCDS (M,\alpha)$ are continuous.
As a consequence, the projection map $ev_1$ is also open.
Indeed, let $\U$ be an open set in $\TCDS (M,\alpha)$, then
	\[ ev^{-1} \left( \ev_1 (\U) \right) = \bigcup \, (\Phi_H, H, h) \circ \U = \bigcup \, \U \circ (\Phi_H, H, h) \]
is open, where the unions are taken over all topological contact dynamical systems $(\Phi_H, H , h)$ with time-one map the identity.
Therefore $\ev_1 (\U)$ is open by definition of the quotient topology.
That makes $\Homeo (M,\xi)$ a topological group with respect to the contact topology, and moreover, the projection of a first countable neighborhood basis at an element $(\Phi, H, h)$ of $\TCDS (M,\alpha)$ defines a first countable neighborhood basis of the contact topology on $\Homeo (M,\xi)$ at $\phi_H^1$.
\end{proof}

In fact, the same proofs verify that the two topological groups of topological Hamiltonian dynamical systems and of Hamiltonian homeomorphisms with the Hamiltonian topologies \cite{mueller:ghh07, mueller:ghh08, mueller:ghl08} admit left invariant metrics.

\begin{thm}
The group of contact homeomorphisms with the contact topology is path-connected and locally-path-connected.
Any two contact homeomorphisms can be connected inside $\Homeo (M,\xi)$ by a topological contact isotopy.
\end{thm}

\begin{proof}
To prove path-connectedness, it suffices to show every $\phi \in \Homeo (M,\xi)$ can be connected to the identity by a path $\ell \colon [0,1] \to \Homeo (M,\xi)$ that is continuous with respect to the contact topology, and with $\ell (0) = \id$ and $\ell (1) = \phi$.
By definition, there exists a topological contact dynamical system $(\Phi_H, H, h)$ with time-one map $\phi_H^1 = \phi$.
Define $\ell (s) = \phi_H^s$.
By Lemma~\ref{lem:rep}, this path factors through the group of topological contact dynamical systems via the evaluation map $\ev_1$, and it suffices to show that the function $s \mapsto (\Phi_{H^s}, H^s, h^s)$ is continuous with respect to the contact metric on $\TCDS (M,\alpha)$.

Let $\epsilon > 0$.
It clearly suffices to show that $\dbar (\Phi_{H^r}, \Phi_{H^s}) < 3 \epsilon$, $| h^r - h^s | < 3 \epsilon$, and $\| H^r - H^s \| < 3 \epsilon$, provided $| r - s |$ is sufficiently small.
There exists a sequence $(\Phi_{H_i}, H_i, h_i)$ of smooth contact dynamical systems, such that $\dbar (\Phi_H, \Phi_{H_i}) < \epsilon$, $| h - h_i | < \epsilon$, and $\| H - H_i \| < \epsilon$ for $i$ sufficiently large.
Fix such an index $i$.
Then
\begin{align*}
	\dbar (\Phi_{H^r}, \Phi_{H^s}) & \le \dbar (\Phi_{H^r}, \Phi_{H_i^r}) + \dbar (\Phi_{H_i^r}, \Phi_{H_i^s}) + \dbar (\Phi_{H_i^s}, \Phi_{H^s}) \\
	& \le \dbar (\Phi_H, \Phi_{H_i}) + \dbar (\Phi_{H_i^r}, \Phi_{H_i^s}) + \dbar (\Phi_{H_i}, \Phi_H) \\
	& < \dbar (\Phi_{H_i^r}, \Phi_{H_i^s}) + 2 \epsilon,
\end{align*}
and similarly for the contact Hamiltonians and conformal factors.
It only remains to show that $\dbar (\Phi_{H_i^r}, \Phi_{H_i^s}) < \epsilon$, $| h_i^r - h_i^s | < \epsilon$, and $\| H_i^r - H_i^s \| < \epsilon$, provided $| r - s |$ is sufficiently small, where the index $i$ is \emph{fixed}.

Since the maps $(t, x) \mapsto \phi_{H_i}^t (x)$ and $(t, x) \mapsto (\phi_{H_i}^t)^{-1} (x)$ are continuous, the first inequality is obvious.
Moreover, $h_i^r (t, x) = h_i (r t, x)$, and $h_i$ is continuous, so that $| h_i (r t, x) - h_i (s t, x) | < \epsilon$, provided $| r - s |$ is sufficiently small.
The function $H_i$ is also Lipschitz continuous, so that there exists a constant $C$ that satisfies
	\[ \| H_i^r - H_i^s \|_\oneinfty \le \| H_i^r - H_i^s \|_\infty \le C | r - s | < \epsilon, \]
provided $| r - s |$ is sufficiently small.
See Lemma~\ref{lem:zeta-est} below for an explicit constant $C$.
That proves $\ell$ is continuous with respect to the contact topology, and the proof of path-connectedness is complete.

It suffices to prove local path-connectedness at the identity.
If $(\Phi_H, H, h)$ is a topological contact dynamical system within $\epsilon$-distance of $(\id, 0, 0)$, then so is $(\Phi_{H^s}, H^s, h^s)$ for every $s \in [0,1]$.
Thus the $d_\alpha$-metric balls of radius $\epsilon > 0$ that are centered at the identity $(\id, 0, 0) \in \TCDS (M,\alpha)$ define a basis of path-connected neighborhoods.
Their projections via the time-one evaluation map $\ev_1$ form a basis of path-connected neighborhoods at the identity in $\Homeo (M,\xi)$.
The proof of local path-connectedness is now verbatim the same as for path-connectedness.

Note that the path $\ell$ chosen above is in fact a topological contact isotopy.
By replacing $H$ by $H^\zeta$, and $H_i$ by $H_i^\zeta$, for a \emph{fixed} reparameterization function $\zeta \colon [0,1] \to [0,1]$, with $\zeta = 0$ near $t = 0$ and $\zeta = 1$ near $t = 1$, any $\phi \in \Homeo (M,\xi)$ can be connected to the identity by a boundary flat topological contact isotopy.
The concatenation of two such topological contact isotopies is again a topological contact isotopy, proving the final statement of the theorem.
\end{proof}

The subspace topology on the subset of topological strictly contact dynamical systems $\TSCDS (M,\alpha) \subset \TCDS (M,\alpha)$ is also called the contact topology, and this topology is induced by the restriction of the contact metric $d_\alpha$ to $\TSCDS (M,\alpha)$.
The time-one evaluation map
	\[ \ev_1 \colon \TSCDS (M,\alpha) \rightarrow \Homeo (M,\alpha), \ (\Phi_H, H) \mapsto \phi_H^1 \]
is by definition surjective, and induces a quotient topology on $\Homeo (M,\alpha)$, which is also called the contact topology.
Again the evaluation map $\ev_1$ is a continuous homomorphism, and the topological spaces $\TSCDS (M,\alpha)$ and $\Homeo (M,\alpha)$ admit left invariant metrics.
The proofs of the following statements are the same as in the case of contact homeomorphisms.

\begin{thm}
The projection map $ev_1$ induces the structure of a first countable topological group on the topological space $\Homeo (M,\alpha)$ with the contact topology, and $\Homeo (M,\alpha)$ is a topological subgroup of $\Homeo (M,\xi)$.
\end{thm}

\begin{cor}
The space $\Homeo (M,\alpha)$ equipped with the contact topology admits a left invariant metric that generates its topology.
\end{cor}

\begin{thm}
The group $\Homeo (M,\alpha)$ of strictly contact homeomorphism with the contact topology is path-connected and locally path-connected.
Any two strictly contact homeomorphisms can be connected inside $\Homeo (M,\alpha)$ by a topological strictly contact isotopy.
\end{thm}

The proofs of the corresponding results in this section for continuous contact and strictly contact dynamical systems are similar and thus omitted.

\section{Regularization} \label{sec:regular}
A contact isotopy $\Phi_H = \{ \phi_H^t \}$ is called \emph{regular} if for every time $t$, its generating Hamiltonian $H_t$ is not identically zero.

\begin{pro} \label{pro:regular}
Let $\Phi_H$ be a contact isotopy generated by a smooth Hamiltonian $H \colon S^1 \times M \to \R$.
Then there exists an arbitrarily small (in the $C^\infty$-topology) contact loop $\Phi_F$ generated by a smooth Hamiltonian $F \colon S^1 \times M \to \R$, such that the isotopy $\Phi_F^{-1} \circ \Phi_H$ is regular.
\end{pro}

In particular, the isotopy $\Phi_F^{-1} \circ \Phi_H$ can be chosen arbitrarily close to the isotopy $\Phi_H$ in the contact metric of Definition~\ref{dfn:contact-metric}.
The proof is an adaptation of the one given in \cite{polterovich:ggs01} for Hamiltonian isotopies, and is divided into three steps.
Recall that by Lemma~\ref{lem:contact-ham}, the isotopy $\Phi_F^{-1} \circ \Phi_H$ is generated by the Hamiltonian that at time $t$ is given by
	\[ (\overline{F} \# H)_t = e^{- f_t} \cdot \left( (H_t - F_t) \circ \phi_F^t \right), \]
so that $\Phi_F^{-1} \circ \Phi_H$ is regular if and only if for every time $t$ the function $H_t - F_t$ is not identically zero.

\begin{proof}
Step 1.
Consider a collection of smooth functions $G^1, \ldots, G^{2 k} \colon S^1 \times M \to \R$, $1 \le k \le n$, with the property
	\[ \int_0^1 G^j (t,x) \, dt = 0 \]
for every $x \in M$.
Here and in the following, we identify $S^1$ with $\R / \Z$.
For example, one may choose $G^j (t,x) = f_j (t) \cdot g_j (x)$ for smooth functions $f_j \colon S^1 \to \R$ with mean value zero, and $g_j \colon M \to \R$.
Then for a fixed $t \in S^1$, define $\phi_{t,\epsilon}^j \in \Diff (M,\xi)$ as the time-$\epsilon$ map of the contact isotopy generated by the \emph{time-independent} Hamiltonian $\int_0^t G_s^j \, ds$.
If $\epsilon = (\epsilon_1, \ldots, \epsilon_{2 k})$ is a vector in $\R^{2 k}$, the composition
	\[ \phi_{t,\epsilon} = \phi_{t,\epsilon_1}^1 \circ \cdots \circ \phi_{t,\epsilon_{2 k}}^{2 k} \in \Diff (M,\xi) \]
defines a \emph{$2 k$-parameter variation of the constant loop}, that is, a smooth family of contact loops $\{ \phi_{t,\epsilon} \}_{0 \le t \le 1}$ with $\phi_{0,\epsilon} = \phi_{1,\epsilon} = \id$, and $\phi_{t,0} = \id$ for all $t$.
Denote by $F_\epsilon \colon S^1 \times M \to \R$ the smooth Hamiltonian generating the isotopy $\{ \phi_{t,\epsilon} \}_{0 \le t \le 1}$, and let $F_{t,\epsilon} = F_\epsilon (t, \cdot)$.
We need the following lemma.

\begin{lem} \label{lem:two-par-fam}
Let $\phi_{s,t} \in \Diff (M,\xi)$ be a smooth two-parameter family of contact diffeomorphisms with $\phi_{0,0} = \id$, and denote by $H = \{ H_{s,t} \}$ and $F = \{ F_{s,t} \}$ the families of smooth functions on $M$ that generate the contact isotopies $\{ \phi_{s,t} \}_{0 \le t \le 1}$ and $\{ \phi_{s,t} \}_{0 \le s \le 1}$, respectively.
Then
\begin{align} \label{eqn:two-par-fam-fctns}
	\frac{d}{ds} H_{s,t} = \frac{d}{dt} F_{s,t} - \{ H_{s,t}, F_{s,t} \},
\end{align}
where the \emph{Poisson bracket} is defined by $\{ H, F \} = - \alpha ([X_H, X_F])$.
\end{lem}

\begin{proof}
If $\{ \phi_{s,t} \}$ is a two-parameter family of diffeomorphisms with $\phi_{0,0} = \id$, write
	\[ X_{s,t} = \left( \frac{d}{dt} \phi_{s,t} \right) \circ \phi_{s,t}^{-1}, \ \ \mbox{and} \ \ Y_{s,t} = \left( \frac{d}{ds} \phi_{s,t} \right) \circ \phi_{s,t}^{-1}. \]
By Proposition~I.1.1 in \cite{banyaga:sgd78},
\begin{align} \label{eqn:two-par-fam-vfs}
	\frac{d}{ds} X_{s,t} = \frac{d}{dt} Y_{s,t} + [X_{s,t}, Y_{s,t}].
\end{align}
Since $\phi_{s,t} \in \Diff (M,\xi)$ for all $s$ and $t$, $X_{s,t}$ and $Y_{s,t}$ are contact vector fields, and $H_{s,t} = \alpha (X_{s,t})$ and $F_{s,t} = \alpha (Y_{s,t})$.
Contracting $\alpha$ with both sides of (\ref{eqn:two-par-fam-vfs}), and observing that $\alpha$ is independent of $s$ and $t$, proves (\ref{eqn:two-par-fam-fctns}).
\end{proof}

A similar result also holds for smooth two-parameter families of Hamiltonian diffeomorphisms, see \cite{banyaga:sgd78, banyaga:scd97, polterovich:ggs01}.
Namely, if $\widehat{H} = \{ \widehat{H}_{s,t} \}$ and $\widehat{F} = \{ \widehat{F}_{s,t} \}$ are smooth Hamiltonians generating a two-parameter family $\{ \widehat{\phi}_{s,t} \}$ of Hamiltonian diffeomorphisms, then
\begin{align} \label{eqn:two-par-fam-hams}
	\frac{d}{ds} \widehat{H}_{s,t} = \frac{d}{dt} \widehat{F}_{s,t} - \{ \widehat{H}_{s,t}, \widehat{F}_{s,t} \},
\end{align}
where this time the Poisson bracket is given by $\{ \widehat{H}, \widehat{F} \} = \omega (X_{\widehat{H}}, X_{\widehat{F}})$.
In fact, the proof of Lemma~\ref{lem:two-par-fam} given above follows the same line of argument as in the Hamiltonian case.
Alternatively, one can lift the two-parameter family $\{ \phi_{s,t} \}$ of Lemma~\ref{lem:two-par-fam} to a two-parameter family $\{ \widehat{\phi}_{s,t} \}$ of Hamiltonian diffeomorphisms on the symplectization $(M \times \R, - d (e^\theta \alpha))$ of $(M,\alpha)$, such that
	\[ \widehat{H}_{s,t} (x, \theta) = e^\theta H_{s,t} (x), \ \ \mbox{and} \ \ \widehat{F}_{s,t} (x, \theta) = e^\theta F_{s,t} (x). \]
With our sign conventions, $\{ \widehat{H}_{s,t}, \widehat{F}_{s,t} \} = e^\theta \left\{ H_{s,t}, F_{s,t} \right\}$, and thus equation (\ref{eqn:two-par-fam-fctns}) also follows from the Hamiltonian version (\ref{eqn:two-par-fam-hams}) of the lemma.

Returning to the proof of Proposition~\ref{pro:regular}, observe that $F_0 = 0$, and therefore Lemma~\ref{lem:two-par-fam} yields
	\[ \left. \frac{d}{d\epsilon_j} \right|_{\epsilon = 0} F_{t,\epsilon} = G_t^j. \]

Step 2.
Fix a point $p \in M$, and consider an even dimensional linear subspace $E^{2 k} \subset T_p^* M$, where again $1 \le k \le n$.
Choose a basis $\{ u_1, \ldots, u_k, v_1, \ldots v_k \}$ of $E$, and for $1 \le j \le k$, define
	\[ \gamma_j (t) = \cos (2 \pi t) u_j + \sin (2 \pi t) v_j, \ \ \mbox{and} \ \ \gamma_{k + j} (t) = - \cos (2 \pi t) u_j + \sin (2 \pi t) v_j. \]
The vectors $\gamma_1 (t), \ldots, \gamma_{2 k} (t)$ are linearly independent for each $t$, and $\int_0^1 \gamma_j (t) \, dt = 0$.
By the last equality, a collection of functions $G^1, \ldots, G^{2 k}$ as in step 1 can be chosen so that $dG_t^j (p) = - \gamma_j (t)$.
For example, in local coordinates let $u_j$ correspond to $(dx_j)_p$ and $v_j$ to $(dy_j)_p$, and define locally
	\[ G_t^j = - \cos (2 \pi t) x_j - \sin (2 \pi t) y_j, \ \ \mbox{and} \ \ G_j^{k + j} = \cos (2 \pi t) x_j - \sin (2 \pi t) y_j. \]
Then cut off the functions $x_j$ and $y_j$ in a neighborhood of $p$ in order to obtain globally defined smooth functions $G_t^j$ on $M$.

Step 3. Define a mapping $I \colon S^1 \times \R^{2 k} \to E^{2 k}$ by $(t, \epsilon) \mapsto d (H_t - F_{t,\epsilon}) (p)$.
Since
	\[ \left. \frac{d}{d\epsilon_j} \right|_{\epsilon = 0} I (t,\epsilon) = \gamma_j (t), \]
$I$ is a submersion into a neighborhood $U$ of the circle $\{ \epsilon = 0 \}$.
Denote the restriction of $I$ to $S^1 \times U$ by $J$.
Then $J^{-1} (0)$ is a one-dimensional submanifold of $S^1 \times U$, so its projection to $U$ is nowhere dense.
Hence there exist arbitrarily small values of the parameter $\epsilon$ such that $d (H_t - F_{t,\epsilon}) (p) \not= 0$ for all $t$.
The contact isotopy generated by such a smooth Hamiltonian $\overline{F_\epsilon} \# H$ is regular, and the proof of Proposition~\ref{pro:regular} is complete.
\end{proof}

Suppose that in Proposition~\ref{pro:regular}, the isotopy $\Phi_H$ is strictly contact, or in other words, its Hamiltonian is basic, and one tries to find a Hamiltonian $F$ as in the same proposition that is also basic.
The difficulty in adapting the above proof is that in general, it is not possible to extend a locally defined basic function on standard contact $\R^{2n + 1}$ to a \emph{basic} function on $M$.

\begin{exa}[\cite{mueller:vfs11}] \label{exa:three-torus}
Let $T^3$ be the three-dimensional torus equipped with the contact form $\alpha = \cos z \, dx - \sin z \, dy$, where $x, y, z \in \R / 2 \pi \Z$ are coordinates on $T^3$.
A basic function on $(T^3,\alpha)$ is independent of $x$ and $y$.
Thus given time-dependent basic functions $G^1$ and $G^2$ on $T^3$, and a point $p \in T^3$, the cotangent vectors $dG_t^j (p) = (\frac{\partial}{\partial z} G^j) (t,p) \, dz$ are linearly dependent in $T_p^* T^3$ for all $t$.
\end{exa}

If $\alpha$ is regular, then cutting off a basic function in a neighborhood of a Reeb orbit is always possible.
Indeed, simply lift a cut-off function from the base of the associated Boothby-Wang prequantization bundle to the total space.

The last step of the proof of Proposition~\ref{pro:regular} required the existence of at least a two-parameter variation of the constant loop.
One parameter can be taken to be a perturbation in the direction of the Reeb flow.
However, a one-dimensional subspace of $E$ as in Example~\ref{exa:three-torus} does not possess a basis $\{ \gamma \}$ with $\int_0^1 \gamma (t) \, dt = 0$.

\begin{pro} \label{pro:strictly-regular}
Let $\Phi_H$ be a contact isotopy generated by a smooth Hamiltonian $H \colon S^1 \times M \to \R$.
Then there exists an arbitrarily small (in the $C^\infty$-topology) contact loop $\Phi_F$ generated by a smooth \emph{basic} Hamiltonian $F \colon S^1 \times M \to \R$, and finitely many points $t_1, \ldots, t_k$, such that the isotopy $\Phi_F^{-1} \circ \Phi_H$ is regular away from the points $t_i$, i.e.\ the smooth function $H_t - F_t$ is not identically zero for all $t \notin \{ t_1, \ldots, t_k \}$.
In fact, the function $F$ can be chosen so that $F (t, x) = f (t)$ for a smooth function $f \colon S^1 \to \R$, and given any subset $T \subset S^1$ with empty interior, we may impose $t_i \notin T$ for all $i = 1, \ldots, k$.
\end{pro}

\begin{proof}
Fix a point $p \in M$.
There exists an arbitrarily $C^\infty$-small smooth function $f \colon S^1 \to \R$ with $\int_{S^1} f \, dt = 0$, such that the smooth map $S^1 \to \R$ defined by $t \mapsto H (t, p) - f (t)$ has only finitely many zeros that occur away from the subset $T \subset S^1$.
The basic Hamiltonian $F$ defined by $F_t = f (t) \cdot 1$ generates a loop by formula (\ref{eqn:loop}) in the beginning of the next section.
\end{proof}

\section{Reparameterization of contact isotopies} \label{sec:rep}
Suppose $H \colon [0,1] \times M \to \R$ is a smooth Hamiltonian, generating the contact isotopy $\Phi_H = \{ \phi_H^t \}$, $a < b$ are real numbers, and $\zeta \colon [a,b] \to [0,1]$ is a smooth function.
The reparameterized isotopy
	\[ \Phi_{H^\zeta} = \{ \phi_{H^{\zeta}}^t \}_{a \le t \le b} = \{ \phi_H^{\zeta (t)} \}_{a \le t \le b} \]
is generated by the Hamiltonian $H^\zeta \colon [a,b] \times M \to \R$, defined by
\begin{align} \label{eqn:rep}
	H^\zeta (t, x) = \zeta' (t) \cdot H (\zeta(t), x),
\end{align}
where $\zeta'$ denotes the derivative of $\zeta$.
In particular, $\Phi_{H^\zeta}$ is a loop if and only if
\begin{align} \label{eqn:loop}
	\zeta (b) - \zeta (a) = \int_a^b \zeta' (t) \, dt = 0.
\end{align}
If $\zeta (a) = 0$, $\zeta (b) = 1$, and the function $\zeta$ is monotone, then the reparameterized isotopy traverses the same path as the original isotopy at different speed, and we refer to the function $\zeta$ as a \emph{reparameterization function}.
In the special case $\zeta (t) = s t$ for a real number $s$, we also write $H^\zeta = H^s$ as in sections~\ref{sec:contact-homeo} and \ref{sec:contact-topo}.
Since $\phi_{H^\zeta}^t = \phi_H^{\zeta (t)}$, the conformal factor $h^\zeta$ of the isotopy $\Phi_{H^\zeta}$ is given by $h_t^\zeta = h_{\zeta (t)}$.
This also follows from a change of variables in equation (\ref{eqn:conformal-factor}).
In particular, if $\Phi_H$ is strictly contact, then so is the reparameterized isotopy $\Phi_{H^\zeta}$.

We state a series of useful lemmas.
The proofs are straightforward and similar to the Hamiltonian case in \cite{mueller:ghh07, mueller:ghh08, mueller:ghl08}, and thus are omitted.

\begin{lem} \label{lem:zeta-est}
Let $H \colon [0,1] \times M \to \R$ be a smooth Hamiltonian function, and $\zeta_1, \zeta_2 \colon [0,1] \to [0,1]$ be two smooth functions.
Then
	\[ \osc \left( H_t^{\zeta_1} - H_t^{\zeta_2} \right) \le 2 L \cdot | \zeta_1' (t) | \cdot | \zeta_1 (t) - \zeta_2 (t) | + | \zeta_1' (t) - \zeta_2' (t) | \cdot \osc \left( H_{\zeta_2 (t)} \right), \]
and
	\[ \left| c \left( H_t^{\zeta_1} - H_t^{\zeta_2} \right) \right| \le L \cdot | \zeta_1' (t) | \cdot | \zeta_1 (t) - \zeta_2 (t) | + | \zeta_1' (t) - \zeta_2' (t) | \cdot \left| c \left( H_{\zeta_2 (t)} \right) \right|, \]
for all $0 \le t \le 1$, where $L$ is a Lipschitz constant that depends only on $H$.
If in addition $\zeta_1$ is monotone, then
	\[ \left\| H^{\zeta_1} - H^{\zeta_2} \right\|_\oneinfty \le 3 L \cdot \max_{0 \le t \le1} | \zeta_1 (t) - \zeta_2 (t) | + \| H \|_\infty \cdot \int_0^1 | \zeta_1' (t) - \zeta_2' (t) | \, dt. \]
\end{lem}

An isotopy $\{ \phi_t \}_{a \le t \le b} $ is called \emph{boundary flat} if it is constant near the two end points, i.e.\ there exists a constant $\delta > 0$ such that $\phi_t = \phi_a$ for $t - a < \delta$, and $\phi_t = \phi_b$ for $b - t < \delta$.
In terms of the Hamiltonian of a contact isotopy $\Phi_H = \{ \phi_t \}$, this is equivalent to $H_t = 0$ for $t - a < \delta$ and $b - t < \delta$.

Given a contact isotopy $\Phi_H$, and a reparameterization function $\zeta \colon [0,1] \to [0,1]$ so that $\zeta = 0$ near $t = 0$ and $\zeta = 1$ near $t = 1$, the reparameterized isotopy $\Phi_{H^\zeta}$ is boundary flat.
Choosing the function $\zeta$ appropriately proves the following lemma.

\begin{lem}[Approximation by boundary flat contact isotopies] \label{lem:approx-bdy-flat}
Given a smooth Hamiltonian $H \colon [0,1] \times M \to \R$, and $\epsilon > 0$, there exists a reparameterization function $\zeta \colon [0,1] \to [0,1]$, such that $H^\zeta$ is boundary flat, and
	\[ \| H - H^\zeta \|_\oneinfty < \epsilon, \]
	\[ \max (\| H_0 \|, \| H_1 \|) \le \| H - H^\zeta \|_\infty < \max (\| H_0 \|, \| H_1 \|) + \epsilon, \ \mbox{and} \]
	\[ \dbar (\Phi_H, \Phi_{H^\zeta}) < \epsilon, \ \ \mbox{and} \ \ | h - h^\zeta | < \epsilon, \]
where $H_t = H (t, \cdot)$ for $ t = 0$ and $t = 1$.
Moreover, $\| H^\zeta \|_\oneinfty = \| H \|_\oneinfty$ and $\| H^\zeta \|_\infty < \| H \|_\infty + \epsilon$.
In particular, the two end points $\phi_H^{\zeta (0)} = \phi_H^0$ and $\phi_H^{\zeta (1)} = \phi_H^1$ coincide, and $H^\zeta$ can be extended to a function on $\R \times M$ that is $1$-periodic in time.
\end{lem}

For later reference, we consider the following reparameterization, which is useful to concatenate boundary flat contact isotopies.
Given $a < b$, and a Hamiltonian $H$ defined on $[0,1] \times M$, denote by $\zeta_{a,b} \colon [a,b] \to [0,1]$ the unique linear function with $\zeta (a) = 0$ and $\zeta (b) = 1$, and by $H^{a,b}$ the reparameterized Hamiltonian defined on $[a,b] \times M$.
Of course, if $\Phi_H$ is boundary flat, then so is its reparameterization.
The two norms satisfy $\| H^{a,b} \|_\oneinfty = \| H \|_\oneinfty$, and
\begin{align} \label{eqn:interval-change}
	\| H^{a,b} \|_\infty = \frac{1}{b - a} \| H \|_\infty.
\end{align}

The main ingredient in the proof of Lemma~\ref{lem:main} is the following result.
This is where we apply the regularization procedure established in the previous section.
Cf.\ \cite{mueller:ghh08, mueller:ghl08}.

\begin{lem} \label{lem:tech-lem}
Let $H \colon [0,1] \times M \to \R$ be a smooth Hamiltonian, generating the contact isotopy $\Phi_H = \{ \phi_H^t \}$, with $(\phi_H^t)^* \alpha = e^{h_t} \alpha$, and let $\epsilon > 0$.
Then there exists a smooth Hamiltonian $F \colon [0,1] \times M \to \R$, with $(\phi_F^t)^* \alpha = e^{f_t} \alpha$, such that the isotopy $\Phi_F$ coincides with a reparameterization of the isotopy $\Phi_H$, up to a small reparameterization with respect to the $L^\oneinfty$-contact metric followed by a $C^\infty$-small perturbation, and
\begin{enumerate}
	\item [\emph{(i)}] the end points of the isotopies coincide, i.e.\ $\phi_F^0 = \phi_H^0$ and $\phi_F^1 = \phi_H^1$,
	\item [\emph{(ii)}] the norms satisfy the inequalities $\| F \|_\oneinfty \le \| F \|_\infty < \| H \|_\oneinfty + \epsilon$,
	\item [\emph{(iii)}] the distances of the two isotopies to their common left end point satisfy the inequality $\dbar (\Phi_F, \phi_F^0) < \dbar (\Phi_H, \phi_H^0) + \epsilon$, and
	\item [\emph{(iv)}] for the conformal factors, we have $| f - f_0 | < | h - h_0 | + \epsilon$.
\end{enumerate}
In (iii), $\phi_F^0 = \phi_H^0$ denotes the constant isotopy $t \mapsto \phi_F^0 = \phi_H^0$, and $f_0 = h_0$ its conformal factor in (iv).
\end{lem}

\begin{proof}
By Lemma~\ref{lem:approx-bdy-flat}, we may assume that the isotopy $\Phi_H$ is boundary flat, and its Hamiltonian can be considered as a function from $S^1 \times M$ to $\R$.
Consider the isotopy $\{ \phi_G^t \} = \{ \phi_t \circ \phi_H^t \}$, where $\{ \phi_t \}$ is a loop in $\Diff (M,\xi)$ with $\phi_0 = \phi_1 = \id$.
Clearly $\phi_G^0 = \phi_H^0$ and $\phi_G^1 = \phi_H^1$.
By Proposition~\ref{pro:regular}, we may choose the loop $\{ \phi_t \}$ so that it is arbitrarily close to the constant loop $\id$ in the $C^\infty$-metric, and in particular, its generating Hamiltonian is arbitrarily small in the $L^\oneinfty$-norm, and its conformal factor is arbitrarily close to zero.
Moreover, $\| G_t \| \not= 0$ for all $t \in S^1$.
Therefore $\dbar (\Phi_G, \phi_G^0) < \dbar (\Phi_H, \phi_H^0) + \epsilon$, $| g - g_0 | < | h - h_0 | + \epsilon$, and
	\[ \| G \|_\oneinfty < \| H \|_\oneinfty + \frac{\epsilon}{2}. \]

Consider the Hamiltonian $G^\zeta$, where $\zeta$ is the \emph{inverse} (here we use $\| G_t \| \not= 0$ for all $t$) of the function
\begin{align} \label{eqn:eta}
	\eta \colon [0,1] \longrightarrow [0,1], \ t \longmapsto \frac{\int_0^t \| G_s \| \, ds}{\int_0^1 \| G_s \| \, ds}.
\end{align}
Then $\zeta$ fixes $0$ and $1$, so that $\Phi_G^\zeta$ has the same end points as $\Phi_G$.
By the chain rule,
	\[ \zeta' (t) = \frac{\int_0^1 \| G_s \| \, ds}{\| G_{\zeta (t)} \|}, \]
hence for every $t$,
	\[ \| G_t^\zeta \| = \zeta' (t) \cdot \| G_{\zeta (t)} \| = \int_0^1 \| G_s \| \, ds = \| G \|_\oneinfty. \]
Therefore
	\[ \| G^\zeta \|_\infty = \| G \|_\oneinfty < \| H \|_\oneinfty + \frac{\epsilon}{2}. \]
The function $\zeta$ may only be $C^1$ but not $C^\infty$-smooth.
We approximate $\zeta$ in the $C^1$-topology by a smooth diffeomorphism $\rho$ of $[0,1]$ that also fixes $0$ and $1$, to obtain a smooth Hamiltonian $F = G^\rho$, with $\| F \|_\infty < \| G^\zeta \|_\infty + \frac{\epsilon}{2}$.
Then $F$ clearly satisfies (i) and (ii).
Since $\Phi_F$ is just a reparameterization of $\Phi_G$, we also have
	\[ \dbar (\Phi_F, \phi_F^0) = \dbar (\Phi_G, \phi_G^0) < \dbar (\Phi_H, \phi_H^0) + \epsilon, \]
and similarly $| f - f_0 | < | h - h_0 | + \epsilon$ for the conformal factors.
\end{proof}

\begin{lem} \label{lem:strictly-tech-lem}
Let $H \colon [0,1] \times M \to \R$ be a smooth basic Hamiltonian, generating the strictly contact isotopy $\Phi_H = \{ \phi_H^t \}$, and let $\epsilon > 0$.
Then there exists a smooth basic Hamiltonian $F \colon [0,1] \times M \to \R$ such that
\begin{enumerate}
	\item [\emph{(i)}] the end points of the isotopies coincide, i.e.\ $\phi_F^0 = \phi_H^0$ and $\phi_F^1 = \phi_H^1$,
	\item [\emph{(ii)}] the norms satisfy the inequalities $\| F \|_\oneinfty \le \| F \|_\infty < \| H \|_\oneinfty + \epsilon$, and
	\item [\emph{(iii)}] the distances of the two isotopies to their common left end point satisfy the inequality $\dbar (\Phi_F, \phi_F^0) < \dbar (\Phi_H, \phi_H^0) + \epsilon$,
\end{enumerate}
In (iii), $\phi_F^0 = \phi_H^0$ again denotes the constant isotopy $t \mapsto \phi_F^0 = \phi_H^0$.
\end{lem}

\begin{proof}
By Lemma~\ref{lem:approx-bdy-flat}, we may again assume that the isotopy $\Phi_H$ is boundary flat, and its basic Hamiltonian can be considered as a function from $S^1 \times M$ to $\R$.
Let $\{ \phi_G^t \} = \{ \phi_t \circ \phi_H^t \}$, where $\{ \phi_t \}$ is a loop with $\phi_0 = \phi_1 = \id$ in $\Diff (M,\alpha)$.
Clearly $\phi_G^0 = \phi_H^0$ and $\phi_G^1 = \phi_H^1$.
By Proposition~\ref{pro:strictly-regular}, we may choose the loop $\{ \phi_t \}$ so that it is arbitrarily close to the constant loop $\id$ in the $C^\infty$-metric, and $\| G_t \| \not= 0$ except at finitely many points $0 < t_1 < \ldots < t_k < 1$ in $S^1 = \R / \Z$.
In particular, $\dbar (\Phi_G, \phi_G^0) < \dbar (\Phi_H, \phi_H^0) + \epsilon$ and $\| G \|_\oneinfty < \| H \|_\oneinfty + \frac{\epsilon}{3}$.
For convenience, denote $t_0 = 0$ and $t_{k + 1} = 1$.
There exists a constant
	\[ 0 < \delta < \frac{1}{2} \cdot \max_{0 \le i \le k} ( t_{i + 1} - t_i), \]
such that $\| G_t \| < \int_0^1 \| G_s \| \, ds$ if $| t - t_i | < \delta$ for some $1 \le i \le k$, and
	\[ \int_0^1 \| G_s \| \, ds < A = \frac{\int_0^1 \| G_s \| \, ds - \sum_{i = 1}^k \left( \int_{t_i - \delta}^{t_i + \delta} \| G_s \| \, ds \right)}{1 - k \cdot 2 \delta} < \int_0^1 \| G_s \| \, ds + \frac{\epsilon}{3}. \]
In each of the subintervals $[0, t_1 - \delta]$, $[t_i + \delta, t_{i + 1} - \delta]$, and $[t_k + \delta, 1]$, the strictly contact isotopy is regular, and we may reparameterize as in the proof of Lemma~\ref{lem:tech-lem} with the denominator in (\ref{eqn:eta}) replaced by the constant $A$, so that
	\[ \| G_t^{\zeta_i} \| = \zeta'_i (t) \cdot \| G_{\zeta_i (t)} \| = A < \int_0^1 \| G_s \| \, ds + \frac{\epsilon}{3}. \]
That gives rise to a continuous reparameterization function $\zeta \colon [0,1] \to [0,1]$ that fixes the two end points, and is linear with slope equal to one in the intervals $[t_i - \delta, t_i + \delta]$.
Then $\zeta$ is piecewise $C^1$-smooth, and can be approximated by a smooth reparameterization function $\rho \colon [0,1] \to [0,1]$ that fixes both end points, such that
	\[ \| G_t^\rho \| < \| G_t^\zeta \| + \frac{\epsilon}{3} = \zeta' (t) \cdot G_{\zeta (t)} + \frac{\epsilon}{3} \le A + \frac{\epsilon}{3} < \int_0^1 \| G_s \| \, ds + \frac{2 \epsilon}{3}. \]
Therefore $\| G^\rho \|_\infty < \| H \|_\oneinfty + \epsilon$, and setting $F = G^\rho$ completes the proof.
\end{proof}

In fact, the proof goes through without the hypothesis that $H$ is basic.
In that case, the conformal factor $f$ of $F$ satisfies the relation $f_t = h_{\rho (\zeta (t))}$, where $\zeta$ is as in Lemma~\ref{lem:approx-bdy-flat}, and $\rho$ is the smooth function defined in the proof of Lemma~\ref{lem:strictly-tech-lem}.

\section{Proof of the Main Lemma} \label{sec:proof-main}
The main lemma and its proof are inspired by their Hamiltonian counterparts in \cite{mueller:ghh08, mueller:ghl08}.
See these references for a detailed commentary on the proof.

\begin{proof}[Proof of Lemma~\ref{lem:main}]
By definition, there exist smooth contact dynamical systems $(\Phi_{H_i}, H_i, h_i)$, such that
	\[ \dbar (\Phi_H, \Phi_{H_i}) \to 0, \ \ \| H - H_i \|_\oneinfty \to 0, \ \ \mbox{and} \ \ | h - h_i | \to 0 \]
as $i \to \infty$.
Write $\phi_i = \phi_{H_i}^1$ and $\phi = \phi_H^1$.
By Lemma~\ref{lem:approx-bdy-flat}, we may assume without loss of generality that each isotopy $\Phi_{H_i}$ is boundary flat, and their Hamiltonians can be considered as smooth functions $H_i \colon S^1 \times M \to \R$.
We will modify the sequence $(\Phi_{H_i}, H_i, h_i)$ in several steps.
As per usual, a Hamiltonian will be denoted by an upper case Roman letter, and the conformal factor of the generated contact isotopy by the corresponding lower case letter.
For brevity, we often suppress the dependence on the time variable.

Let $\epsilon_i > 0$ be a decreasing sequence of real numbers.
Since $\phi \colon M \to M$ is uniformly continuous, there exists a sequence $\delta_i > 0$, so that $d (\phi (x), \phi (y) ) < \epsilon_i$, for all $x, y \in M$ with $d (x, y) < \delta_i$.
The function $h$ on $M$ is continuous as well, so by making the positive numbers $\delta_i$ smaller if necessary, we may in addition assume $| h (t, x) - h (t, y) | < \epsilon_i$, for all $0 \le t \le 1$, and all $x, y \in M$ with $d (x, y) < \delta_i$.
By again making $\delta_i$ smaller if necessary, we may impose $\delta_i \le \epsilon_i$.

Step 1.
For notational convenience, write $(\Phi_{H_0}, H_0, h_0) = (\id, 0, 0)$.
Then define a sequence $K_i$ of smooth Hamiltonians by
	\[ K_{i + 1} = \left( e^{h_i^1} \cdot (\Hbar_i \# H_{i + 1}) \right) \circ \phi_i^{-1} = \left( e^{h_i^1 - h_i} \cdot \left( (H_{i + 1} - H_i) \circ \Phi_{H_i} \right) \right) \circ \phi_i^{-1} \]
for $i \ge 0$, generating the smooth contact isotopies
	\[ \Phi_{K_{i + 1}} = \phi_i \circ \Phi_{H_i}^{-1} \circ \Phi_{H_{i + 1}} \]
from $\phi_i$ to $\phi_{i + 1}$ (see the remark in section~\ref{sec:contact-geom}), and with conformal factors given by
	\[ k_{i + 1}^t = \left( h_i^1 - h_i^t \right) \circ \left( (\phi_{H_i}^t)^{-1} \circ \phi_{H_{i + 1}}^t \right) + h_{i + 1}^t. \]
Here $\phi_i$ again denotes either the diffeomorphism itself or the corresponding constant isotopy.
By passing to a convergent subsequence of the sequence $(\Phi_{H_i}, H_i, h_i)$ if necessary, we may assume that for all $i$
	\[ \| K_{i + 1} \|_\oneinfty < 3 \cdot e^{| h_i^1 - h_i |} \cdot \| H_{i + 1} - H_i \|_\oneinfty < \epsilon_{i + 1}, \]
	\[ \dbar (\Phi_{H_{i + 1}}, \Phi_{H_i}) < \delta_{i + 1}, \ \ \dbar (\phi, \phi_i) < \epsilon_{i + 1} \ \ \mbox{and} \]
	\[ | h - h_i | < \epsilon_{i + 1}. \]
In the first line, we have used the straightforward inequalities $| \cdot | \le \| \cdot \| < 3 \, | \cdot |$ from \cite{ms:tcd11} for \emph{time-independent} functions on $M$, and that the sequence $| h_i^1 - h_i |$ is bounded.
Then
\begin{align*}
	\dbar ( \Phi_{K_{i + 1}}, \phi_i ) & = d (\Phi_{K_{i + 1}}, \phi_i) + d (\Phi_{K_{i + 1}}^{-1}, \phi_i^{-1}) \\
	& \le d (\phi_i, \phi) + d (\phi \circ \Phi_{H_i}^{-1} \circ \Phi_{H_{i + 1}}, \phi) + d (\phi, \phi_i)+ d (\Phi_{H_{i + 1}}^{-1} \circ \Phi_{H_i}, \id) \\
	& \le d (\phi_i, \phi) + d (\phi \circ \Phi_{H_i}^{-1}, \phi \circ \Phi_{H_{i + 1}}^{-1}) + d (\phi, \phi_i)+ d (\Phi_{H_{i + 1}}^{-1}, \Phi_{H_i}^{-1}) \\
	& \le 4 \epsilon_{i + 1}.
\end{align*}
Moreover, the conformal factors of the contact isotopy $\Phi_{K_{i + 1}}$ and of the contact diffeomorphism $\phi_{K_{i + 1}}^0 = \phi_i = \phi_{H_i}^1$ differ by at most
\begin{align*}
	| k_{i + 1} - k_{i + 1}^0 | & \le | h_i^1 - h^1 | + | h^1 \circ ((\Phi_{H_i})^{-1} \circ \Phi_{H_{i + 1}}) - h^1 | + | h^1 - h_i^1 | \\
	& \ \ \ \ + \, | h_i - h | + | h - h \circ ((\Phi_{H_i})^{-1} \circ \Phi_{H_{i + 1}}) | + | h - h_{i + 1} | < 6 \epsilon_{i + 1}.
\end{align*}
In the present situation, we choose $\epsilon_i = \frac{1}{3} \cdot (\frac{1}{2})^{2 i - 1}$.

Step 2.
Applying Lemma~\ref{lem:tech-lem} to each function $K_i$ yields a sequence of smooth Hamiltonians $L_i$, such that the end points of the contact isotopies coincide, that is, $\phi_{L_i}^0 = \phi_{K_i}^0 = \phi_{i - 1}$, $\phi_{L_i}^1 = \phi_{K_i}^1 = \phi_i$, and moreover,
	\[ \| L_i \|_\infty < \| K_i \|_\oneinfty + \epsilon_i < 2 \epsilon_i, \]
	\[ \dbar (\Phi_{L_i}, \phi_{i - 1}) < \dbar (\Phi_{K_i}, \phi_{i - 1}) + \epsilon_i \le 5 \epsilon_i, \ \ \mbox{and} \]
	\[ | l_i - l_i^0 | < | k_i - k_i^0 | + \epsilon_i < 7 \epsilon_i. \]

Step 3.
Using Lemma~\ref{lem:approx-bdy-flat} to reparameterize the Hamiltonians $L_i$, we obtain boundary flat Hamiltonians $M_i$ with the same end points $\phi_{i - 1}$ and $\phi_i$, and
	\[ \| M_i \|_\infty \le \| L_i \|_\infty + \epsilon_i < 3 \epsilon_i. \]
Moreover, since the contact isotopy $\Phi_{M_i}$ is a reparameterization of the isotopy $\Phi_{L_i}$, we have
\begin{align} \label{eqn:dist-constant-iso}
	\dbar (\Phi_{M_i}, \phi_{i - 1}) = \dbar (\Phi_{L_i}, \phi_{i - 1}) \le 5 \epsilon_i, \ \ \mbox{and}
\end{align}
	\[ | m_i - m_i^0 | = | l_i - l_i^0 | < 7 \epsilon_i. \]
For later reference, observe that for the conformal factors $m_i^0 = l_i^0 = k_i^0 = h_{i - 1}^1$.

Step 4.
Let $t_i = 1 - (\frac{1}{2})^i$ for all $i \ge 1$.
In particular, $0 = t_0 < t_1 < t_2 < \ldots < 1$.
Then define a sequence of smooth boundary flat reparameterizations
	\[ N_i = M_i^{t_{i - 1}, t_i} \colon [t_{i - 1}, t_i] \times M \to \R \]
as in section~\ref{sec:rep}.
By equation (\ref{eqn:interval-change}), we have
\begin{align} \label{eqn:decreasing-norm}
	\| N_i \|_\infty = \frac{1}{t_i - t_{i - 1}} \cdot \| M_i \|_\infty = 2^i \cdot \| M_i \|_\infty < 3 \cdot 2^i \cdot \epsilon_i = \frac{1}{2^{i - 1}}.
\end{align}

Step 5.
Define a sequence of smooth contact dynamical systems $(\Phi_{F_i}, F_i, f_i)$ as follows.
Let $\Phi_{F_1} = \Phi_{N_1}$ in the interval $[t_0, t_1]$, and the constant isotopy $\Phi_{F_1} = \phi_1$ in the remaining interval $[t_1, 1]$.
Then for $i > 1$, define recursively
\begin{align*}
	\Phi_{F_i} & = \Phi_{F_{i - 1}} \ \ \mbox{in the interval} \ [0, t_{i - 1}], \\
	\Phi_{F_i} & = \Phi_{N_i} \ \ \mbox{in the interval} \ [t_{i - 1}, t_i], \ \mbox{and} \\
	\Phi_{F_i} & = \phi_i \ \ \mbox{in the interval} \ [t_i, 1].
\end{align*}
The contact isotopies $\Phi_{F_i}$ are obviously continuous.
Due to the boundary flatness of the functions $N_i$, they are in fact smooth.
For $i < j$, the isotopies $\Phi_{F_i}$ and $\Phi_{F_j}$ agree everywhere except in the interval $[t_i, 1]$.
Since both isotopies are constant in the interval $[t_j, 1]$, their maximum distance with respect to the $C^0$-metric is achieved in the interval $[t_i, t_j]$.
In that interval, $\Phi_{F_i}$ is equal to the constant isotopy $\phi_i$, while $\Phi_{F_j}$ at each time coincides with the contact isotopy $\Phi_{N_k}$ from $\phi_{k - 1}$ to $\phi_k$ for some $i < k \le j$.
By equation~(\ref{eqn:dist-constant-iso}),
	\[ \dbar (\Phi_{F_i}, \Phi_{F_j}) = \max_{i < k \le j} \dbar (\phi_i, \Phi_{N_k}) \le \max_{i < k \le j} \dbar (\phi_i, \phi_{k - 1}) + \max_{i < k \le j} \dbar (\phi_{k - 1}, \Phi_{N_k}) < 7 \epsilon_i < \frac{1}{2^i} \]
for $i > 2$, and therefore $\dbar (\Phi_{F_i}, \Phi_{F_j}) \to 0$, as $i, j \to \infty$.

For the sequence $F_i$ of smooth Hamiltonians, we have $F_1 = N_1$ in the interval $[t_0, t_1]$, $F_1 = 0$ in the interval $[t_1, 1]$, and for $i > 1$,
\begin{align*}
		F_i & = F_{i - 1} \ \ \mbox{in the interval} \ [0, t_{i - 1}], \\
		F_i & = N_i \ \ \mbox{in the interval} \ [t_{i - 1}, t_i], \ \mbox{and} \\
		F_i & = 0 \ \ \mbox{in the interval} \ [t_i, 1].
\end{align*}
These Hamiltonians $F_i$ are indeed smooth, due to boundary flatness of the functions $N_i$.
By the same argument as above, for $i < j$
	\[ \| F_i - F_j \|_\infty = \max_{i < k \le j} \| N_k \|_\infty < \frac{1}{2^i} \]
by equation~(\ref{eqn:decreasing-norm}), and thus $\| F_i - F_j \|_\infty \to 0$, as $i, j \to \infty$.

Since the isotopies $\Phi_{F_i}$ and $\Phi_{F_j}$ agree everywhere except in the interval $[t_i, 1]$, and are both constant in the interval $[t_j, 1]$, the difference of their conformal factors also attains its maximum in the interval $[t_i, t_j]$.
In fact,
\begin{align*}
	| f_i - f_j | & \le \max_{i < k \le j} | m_i^1 - m_k | \\
	& \le \max_{i < k \le j} | m_{i + 1}^0 - m_k^0 | + \max_{i < k \le j} | m_k^0 - m_k | \\
	& = \max_{i < k \le j} | h_i^1 - h_{k - 1}^1 | + \max_{i < k \le j} | m_k^0 - m_k | \\
	& < 9 \epsilon_i < \frac{1}{2^i}
\end{align*}
for $i > 2$.

That proves the sequence $(\Phi_{F_i}, F_i, f_i)$ is Cauchy with respect to the $L^\infty$-contact metric, and the limit $(\Phi_F, F, f)$ is a continuous contact dynamical system.
The time-one map is $\phi_F^1 = \lim_i \phi_{H_i}^1 = \phi_H^1$.
By construction, the isotopy $(\Phi_F, F, f)$ is smooth except possibly at time $t = 1$.

To prove the inequalities in equation~(\ref{eqn:small-perturbation}), in the above construction replace the Hamiltonian $F_{i_0}$ by a boundary flattening $H'_{i_0}$ of the Hamiltonian $H_{i_0}$ for a sufficiently large fixed index $i_0$, so that
	\[ \dbar (\Phi_{H'_{i_0}}, \Phi_{H_{i_0}}) < \frac{\epsilon}{3}, \ \ | h'_{i_0} - h_{i_0} | < \frac{\epsilon}{3}, \ \ \mbox{and} \ \ \| H'_{i_0} - H_{i_0} \|_\oneinfty < \frac{\epsilon}{3}. \]
Then
	\[ \dbar (\Phi_H, \Phi_F) \le \dbar (\Phi_H, \Phi_{H_{i_0}}) + \dbar (\Phi_{H_{i_0}}, \Phi_{H'_{i_0}}) + \dbar (\Phi_{H'_{i_0}}, \Phi_F) \le \frac{\epsilon}{3} + \frac{\epsilon}{3} + \sum_{i = i_0}^\infty \frac{1}{2^i} < \epsilon, \]
provided $i_0 = i_0 (\epsilon)$ is chosen sufficiently large, and similarly for the other estimates in equation~(\ref{eqn:small-perturbation}).
For the other inequalities~(\ref{eqn:close-to-id}), instead define $H'_{i_0}$ by applying Lemma~\ref{lem:tech-lem} to the above isotopy $\Phi_{H_{i_0}}$.
\end{proof}

\begin{proof}[Proof of Lemma~\ref{lem:strictly-main}]
The proof is almost verbatim the same as the one of the Main Lemma~\ref{lem:main}, except that in step~2 the reference to Lemma~\ref{lem:tech-lem} is to be replaced by a reference to Lemma~\ref{lem:strictly-tech-lem}.
Then all conformal factors in the construction are zero, and the resulting smooth contact dynamical systems $(\Phi_{F_i}, F_i, 0)$ are strictly contact.
\end{proof}

\section{The energy-capacity inequality and a bi-invariant metric} \label{sec:energy-capacity}
In this section we show how the energy-capacity inequality from \cite{ms:tcd11} gives rise to a bi-invariant metric on the group of strictly contact homeomorphisms.
This metric is similar to the Hofer metric on the groups of Hamiltonian diffeomorphisms \cite{hofer:tps90, lalonde:gse95} and Hamiltonian homeomorphisms \cite{oh:ghh10, mueller:ghh08, mueller:ghl08}, and the bi-invariant metric on the group of strictly contact diffeomorphisms \cite{banyaga:lci06, ms:tcd11}.

\begin{dfn}[Contact energy]
The \emph{contact energy} $E (\phi)$ of a strictly contact homeomorphism $\phi \in \Homeo (M,\alpha)$ is by definition the non-negative number
\begin{align} \label{eqn:energy}
	E (\phi) = \inf_{H \mapsto \phi} \| H \|,
\end{align}
where the infimum is taken over all topological strictly contact dynamical systems $(\Phi, H)$ with time-one map $\phi_H^1 = \phi$.
\end{dfn}

As in the case of the contact energy of strictly contact diffeomorphisms studied in \cite{banyaga:lci06, ms:tcd11}, symmetry and conjugation invariance of the contact energy, as well as the triangle inequality, follow from the group identities and transformation law for topological strictly contact dynamical systems.

\begin{pro}
The contact energy $E (\phi)$ satisfies the following properties.
For $\phi$ and $\psi \in \Homeo (M,\alpha)$, and $\varphi \in \Aut (M,\alpha)$, we have
\begin{align*}
	& \mbox{(symmetry)} & E (\phi^{-1}) = E(\phi), \\
	& \mbox{(conjugation invariance)} & E (\varphi \circ \phi \circ \varphi^{-1}) = E (\phi), \ \mbox{and} \\
	& \mbox{(triangle inequality)} & E (\phi \circ \psi) \le E (\phi) + E (\psi).
\end{align*}
In particular, $E (\phi^{-1} \circ \psi) = E (\psi^{-1} \circ \phi)$ by symmetry, and $E (\phi \circ \psi) = E (\psi \circ \phi)$ by conjugation invariance.
\end{pro}

For the proof of non-degeneracy, recall the following key result from \cite{ms:tcd11}.

\begin{thm}[Energy-capacity inequality \cite{ms:tcd11}] \label{thm:energy-capacity-ineq}
Suppose the time-one map $\phi^1_H \in \Diff_0 (M,\xi)$ of a smooth Hamiltonian $H \colon [0,1] \times M \to \R$ displaces a compact subset $K \subset M$ with non-empty interior.
Then there exists a constant $C > 0$ that depends only on $K$ and $\alpha$, but is independent of the contact isotopy $\{ \phi_H^t \}$, its conformal factor $h \colon [0,1] \times M \to \R$ given by $(\phi_H^t)^* \alpha = e^{h_t} \alpha$, and the Hamiltonian $H$, such that
	\[ \| H \| \ge C e^{- | h |} > 0. \]
In particular, if $\phi \in \Diff_0 (M,\alpha)$, then $\| H \| > C > 0$ for every basic function $H$ that generates the time-one map $\phi_H^1 = \phi$.
\end{thm}

\begin{cor}[Energy-capacity inequality] \label{cor:energy-capacity-ineq}
Suppose $\phi^1_H \in \Homeo (M,\xi)$ is the time-one map of a topological Hamiltonian $H \colon [0,1] \times M \to \R$, and displaces a compact subset $K \subset M$ with non-empty interior.
Then
	\[ \| H \| \ge C e^{- | h |} > 0, \]
where the constant $C = C (K,\alpha) > 0$ is the same as the one in Theorem~\ref{thm:energy-capacity-ineq}.
In particular, this constant is independent of the topological Hamiltonian $H$, and of the topological contact isotopy $\{ \phi_H^t \}$ and topological conformal factor $h \colon [0,1] \times M \to \R$ corresponding to $H$.
If $\phi \in \Homeo (M,\alpha)$, then $\| H \| > C > 0$ for every topological Hamiltonian $H$ of a topological strictly contact isotopy with time-one map $\phi_H^1 = \phi$.
\end{cor}

\begin{proof}
By definition, there exist smooth contact dynamical systems $(\Phi_{H_i}, H_i, h_i)$ that converge with respect to the contact metric $d_\alpha$ to the topological contact dynamical system $(\Phi_H, H, h)$.
Let $\epsilon > 0$.
By compactness of $K$, the time-one map $\phi_{H_i}^1$ displaces $K$ for $i$ sufficiently large, and by Theorem~\ref{thm:energy-capacity-ineq},
	\[ \| H \| > \| H_i \| - \epsilon \ge C e^{- | h_i |} - \epsilon > C e^{- | h | - \epsilon} - \epsilon, \]
for $i$ sufficiently large.
Since $\epsilon$ was arbitrary, the claim follows.
\end{proof}

\begin{cor} \label{cor:non-degeneracy}
The contact energy $E (\phi)$ of a strictly contact homeomorphism $\phi$ vanishes if and only if $\phi = \id$.
\end{cor}

\begin{cor}
The function
	\[ \Homeo (M,\alpha) \times \Homeo (M,\alpha) \rightarrow \R, \ (\phi, \psi) \mapsto E (\phi^{-1} \circ \psi) \]
defines a bi-invariant metric on the group of strictly contact homeomorphisms.
\end{cor}

Again it is possible to replace topological strictly contact dynamical systems by continuous strictly contact dynamical systems in all of the constructions and statements above.
Denote by $E_\oneinfty (\phi) =  \inf \| H \|_\oneinfty$ and $E_\infty (\phi) =  \inf \| H \|_\infty$ the a priori different contact energies that arise by taking the infimum in (\ref{eqn:energy}) over topological strictly contact dynamical systems and continuous strictly contact dynamical systems, respectively.

\begin{thm}
The equality $E_\oneinfty (\phi) = E_\infty (\phi)$ holds for every $\phi \in \Homeo (M,\alpha)$.
\end{thm}

\begin{proof}
The inequality $E_\oneinfty (\phi) \le E_\infty (\phi)$ follows from the definitions, while the reverse inequality $E_\oneinfty (\phi) \ge E_\infty (\phi)$ is a consequence of Lemma~\ref{lem:strictly-main}.
\end{proof}

In fact, the same argument proves that their smooth counterparts $E_\oneinfty$ and $E_\infty$ on the group of strictly contact diffeomorphisms \cite{banyaga:lci06, ms:tcd11}, defined as the infimums over all smooth strictly contact dynamical systems, coincide as well.

\section{Non-smooth contact homeomorphisms} \label{sec:non-smooth}
In \cite{mueller:ghh07, buhovsky:ugh10, banyaga:uch12}, the authors construct Hamiltonian homeomorphisms and strictly contact homeomorphisms that are not Lipschitz continuous, and thus not $C^1$-smooth, on any symplectic manifold and any regular contact manifold.
In this section we generalize these examples to contact homeomorphisms that are not Lipschitz continuous.
The topological contact dynamical system of standard $\R^{2 n + 1}$ constructed below is compactly supported, and by Darboux's theorem, can be considered as a topological contact dynamical system of any given contact form on an arbitrary contact manifold.

The homeomorphisms referenced in the preceding paragraph arise as rotations of a small ball in $\R^{2 n}$ and their lifts to the total space of a prequantization bundle.
The construction in this section is much more involved, due to the fact that a non-trivial contact isotopy that induces a rotation of a ball in $\R^{2 n}$ in the splitting $\R^{2 n} \times \R = \R^{2 n + 1}$ is not compactly supported, and unless the contact form is regular, a locally defined basic Hamiltonian can in general not be extended to a basic function on the entire manifold.
We begin our discussion with an example of a compactly supported smooth contact dynamical system of $\R^{2 n + 1}$ with its standard contact form.

\begin{exa} \label{exa:smooth-iso}
Denote by $(r_1, \ldots, r_n, \theta_1, \ldots, \theta_n, z)$ polar coordinates on $\R^{2 n + 1}$, where $r_i \ge 0$ and $0 \le \theta_i < 2 \pi$, and where $x_i = r_i \cos \theta_i$ and $y_i = r_i \sin \theta_i$ are rectangular coordinates on $\R^{2 n}$.
Let $\xi = \ker \alpha$ be the contact structure defined by the contact form
	\[ \alpha = dz + \frac{1}{2} \sum_{i = 1}^n (x_i \, dy_i - y_i \, dx_i) = dz + \frac{1}{2} \sum_{i = 1}^n r_i^2 \, d\theta_i, \]
which is diffeomorphic to the standard contact form $\alpha_{\rm std} = dz - \sum_{i = 1}^n y_i \, dx_i$.
Since $d\alpha = \sum_{i = 1}^n r_i \, dr_i \wedge d\theta_i$, the contact vector field $X_H$ associated to a smooth time-dependent Hamiltonian $H \colon [0,1] \times \R^{2 n + 1} \to \R$ is in polar coordinates given by
	\[ X_H = \sum_{i = 1}^n \left( \frac{1}{2} r_i \frac{\partial H_t}{\partial z} - \frac{1}{r_i} \frac{\partial H_t}{\partial \theta_i} \right) \frac{\partial}{\partial r_i} + \sum_{i = 1}^n \left( \frac{1}{r_i} \frac{\partial H_t}{\partial r_i} \right) \frac{\partial}{\partial \theta_i} + \left( H - \frac{1}{2} \sum_{i = 1}^n r_i \frac{\partial H_t}{\partial r_i} \right) \frac{\partial}{\partial z}, \]
provided $r_i > 0$ for $i = 1, \ldots, n$.

Let $\rho \colon [0,1] \to \R$ be a smooth function that is identically zero near $r = 1$, and $\eta \colon \R \to [0,1]$ be a compactly supported smooth function with $\eta (z) = 1$ near $z = 0$.
Define a compactly supported autonomous smooth Hamiltonian $H \colon \R^{2 n + 1} \to \R$ by
\begin{align} \label{eqn:ham-exa}
	H (r_1, \ldots, r_n, \theta_1, \ldots, \theta_n, z) = \eta (z) \int_r^1 s \rho (s) ds,
\end{align}
where $r = \sqrt{r_1^2 + \ldots + r_n^2}$, cf.\ \cite{mueller:ghh07, mueller:ghh08, buhovsky:ugh10, banyaga:uch12}.
The contact vector field corresponding to this Hamiltonian function is $X_H = Y_H - Z_H$, where
	\[ Y_H = \frac{1}{2} \eta' (z) \left( \int_r^1 s \rho (s) ds \right) \sum_{i = 1}^n r_i \frac{\partial}{\partial r_i} + \eta (z) \left( \frac{1}{2} r^2 \rho (r) + \int_r^1 s \rho (s) ds \right) \frac{\partial}{\partial z}, \]
and
	\[ Z_H = \eta (z) \rho (r) \sum_{i = 1}^n \frac{\partial}{\partial \theta_i}. \]
Denote by $\{ \phi_Y^t \}$ the smooth isotopy generated by the smooth vector field $Y_H$.
Then $(\phi_Y^t)_* (\frac{\partial}{\partial \theta_i}) = \frac{\partial}{\partial \theta_i}$ for $i = 1, \ldots, n$, and we can express
\begin{align*}
	X_H & = Y_H + \left( \phi_Y^t \right)_* \left( \left( \phi_Y^t \right)^{-1}_* (- Z_H) \right) \\
	& = Y_H + \left( \phi_Y^t \right)_* \left( - \eta \left( \z \circ \left( \phi_Y^t \right)^{-1} \right) \rho \left( \r \circ \left( \phi_Y^t \right)^{-1} \right) \sum_{i = 1}^n \frac{\partial}{\partial \theta_i} \right),
\end{align*}
where the maps $\r$ and $\z \colon \R^{2 n + 1} \to \R$ are defined by $\r (r_1, \ldots, r_n, \theta_1, \ldots, \theta_n, z) = \sqrt{r_1^2 + \ldots + r_n^2}$ and $\z (r_1, \ldots, r_n, \theta_1, \ldots, \theta_n, z) = z$.
Write $\{ \phi_{Y Z}^t \}$ for the smooth isotopy generated by the vector field $(\phi_Y^t)^{-1}_* (- Z_H)$.
The smooth contact isotopy generated by the smooth Hamiltonian $H = H_\rho$ is equal to the composition $\{ \phi_H^t \} = \{ \phi_Y^t \circ \phi_{Y Z}^t \}$.
Given the function $\rho$ one can choose the function $\eta$ so that
\begin{align} \label{eqn:rho-eta}
	\eta \left( z \pm \left( \frac{1}{2} r^2 \rho (r) + \int_r^1 s \rho (s) ds \right) \right) = 1
\end{align}
for $z$ near zero and for all $r$.
Conversely, given $\eta$ one can choose the function $\rho$ so that (\ref{eqn:rho-eta}) holds.
For later reference, we denote by $u = u (\rho, \eta) > 0$ the largest number such that (\ref{eqn:rho-eta}) holds for all $z \in \R$ with $| z | \le u$ and for all $r \ge 0$, and by $U = \R^{2 n} \times [- u, u] \subset \R^{2 n + 1}$ the corresponding neighborhood of the origin in $\R^{2 n + 1}$.
Then on the subset $U \subset \R^{2 n + 1}$, the isotopy is given by
\begin{align}
	& \phi_H^t (r_1, \ldots, r_n, \theta_1, \ldots, \theta_n, z) \label{eqn:iso-exa} \\
	& = \left (r_1, \ldots, r_n, \theta_1 - t \rho (r), \ldots, \theta_n - t \rho (r), z + t \left( \frac{1}{2} r^2 \rho (r) + \int_r^1 s \rho (s) ds \right) \right),\nonumber
\end{align}
and its inverse is $(\phi_{H_\rho}^t)^{-1} = \phi_{H_{-\rho}}^t$.
Denote as usual by $h \colon [0,1] \times \R^{2 n + 1} \to \R$ the conformal factor of the isotopy $\{ \phi_H^t \}$ determined by the identity $(\phi_H^t)^* \alpha = e^{h_t} \alpha$.
By equation~(\ref{eqn:conformal-factor}),
\begin{align} \label{eqn:conformal-factor-exa}
	h_t = \int_0^t \frac{\partial H}{\partial z} \circ \phi_H^s \, ds = \int_0^t \left( \eta' (\z) \int_\r^1 v \rho (v) dv \right) \circ \phi_H^s \, ds,
\end{align}
and the restriction of $h_t$ to $(\phi_H^s)^{-1} (U) \subset \R^{2 n + 1}$ vanishes for all $0 \le s \le 1$.
\end{exa}

For the remainder of this section, let $\rho \colon (0,1] \to \R$ be a smooth function that is identically zero near $r = 1$, and near $r = 0$ coincides with the function $r \mapsto r^{-a}$, where $0 < a < 2$.
Note that in contrast to the situation considered in the preceding example, this function $\rho$ does not extend smoothly or even continuously to the closed interval $[0,1]$.
Given a cut-off function $\eta$ as in the example, we can choose $\rho$ so that (\ref{eqn:rho-eta}) holds for a constant $u = u (\rho, \eta) > 0$, and vice versa.
In order to simplify the subsequent arguments, we assume without loss of generality that $\rho \ge 0$.
Choose a sequence of smooth functions $\rho_j \colon [0,1] \to \R$ indexed by the positive integers, such that $\rho_j (r) = \rho (r)$ for $r \ge \epsilon_j$, and $0 \le \rho_j (r) \le \rho_k (r) \le \rho (r)$, provided $k \ge j$, where $\epsilon_j$ is a decreasing sequence of positive numbers that converges to zero as $j \to \infty$.
Let $H_j = H_{\rho_j}$ be the corresponding sequence of compactly supported autonomous smooth Hamiltonians on $\R^{2 n + 1}$ defined as in equation~(\ref{eqn:ham-exa}).

\begin{lem}
The Hamiltonians $H_j$ converge uniformly to a compactly supported autonomous continuous function $H$ on $\R^{2 n + 1}$.
\end{lem}

\begin{proof}
By construction, the integral
\begin{align} \label{eqn:finite-int}
	\int_{0^+}^1 s \rho (s) ds = \lim_{r \to 0^+} \int_{r}^1 s \rho (s) ds
\end{align}
is finite, and therefore for all $k \ge j$,
	\[ | H_j - H_k | \le | \eta | \int_0^{\epsilon_j} s \left( \rho_k (s) - \rho_j (s) \right) ds \le | \eta | \int_{0^+}^{\epsilon_j} s \rho (s) ds \longrightarrow 0 \]
as $j \to \infty$.
The limit $H$ is defined by (\ref{eqn:ham-exa}) for $r > 0$, and by (\ref{eqn:finite-int}) if $r = 0$.
\end{proof}

Denote by $Y_j = Y_{H_j}$ and $Z_j = Z_{H_j}$ the vector fields defined in Example~\ref{exa:smooth-iso}, so that the smooth contact isotopy $\{ \phi_{H_j}^t \}$ of $\R^{2 n + 1}$ is again given by the composition $\{ \phi_{H_j}^t \} = \{ \phi_{Y_j}^t \circ \phi_{Y_j Z_j}^t \}$ corresponding to the decomposition
	\[ X_{H_j} = Y_j + \left( \phi_{Y_j}^t \right)_* \left( \left( \phi_{Y_j}^t \right)^{-1}_* (- Z_j) \right). \]

\begin{lem} \label{lem:y-conv}
The smooth vector fields $Y_j$ converge to a continuous vector field $Y$ uniformly on $\R^{2 n + 1}$.
In fact, $Y = Y_j = Y_k$ provided $r \ge \epsilon_j$ and $k \ge j$.
In particular, the isotopies $\{ \phi_{Y_j}^t \}$ converge uniformly to a compactly supported continuous isotopy, denoted by $\{ \phi_Y^t \}$, where $\phi_Y^t \colon \R^{2 n + 1} \to \R^{2 n + 1}$ is a continuous map for each time $t$.
\end{lem}

\begin{proof}
By construction, $Y_j = Y_k$ provided $r \ge \epsilon_j$ and $k \ge j$.
If $k \ge j$, then
	\[ | Y_j - Y_k | \le \frac{1}{2} \, | \eta' | \left( \int_{0^+}^{\epsilon_j} s \rho (s) ds \right) n \, \epsilon_j + | \eta | \left( \frac{1}{2} \left. r^2 \rho (r) \right|_{0 < r \le \epsilon_j} + \int_{0^+}^{\epsilon_j} s \rho (s) ds \right) \longrightarrow 0 \]
as $j \to \infty$, uniformly on $\R^{2 n + 1}$.
\end{proof}

\begin{lem} \label{lem:r-bound}
Let
	\[ b = \frac{1}{2} \, | \eta' | \left( \int_{0^+}^1 s \rho (s) ds \right) > 0. \]
Then
	\[ e^{-b} \cdot r \le \left( \r \circ \phi_{Y_j}^t \right) (r_1, \ldots, r_n, \theta_1, \ldots, \theta_n, z) \le e^b \cdot r \]
for all $j$, all $0 \le t \le 1$, and all $(r_1, \ldots, r_n, \theta_1, \ldots, \theta_n, z)$ with $r = \sqrt{r_1^2 + \ldots + r_n^2}$.
\end{lem}

It is crucial to note that the constant $b$ is independent of $j$, $t \in [0,1]$, and $r \ge 0$.

\begin{proof}
The radial components $r_{j, i}$, $i = 1, \ldots, n$, of the flow of the vector field $Y_j$ are solutions to the ordinary differential equations
	\[ \dot r_{j, i} (t) = \frac{1}{2}  \eta' (z_j (t)) \left( \int_{r_j (t)}^1 s \rho_j (s) ds \right)  r_{j, i} (t), \]
where $r_j = \sqrt{r_{j, 1}^2 + \ldots + r_{j, n}^2}$, and $z_j (t)$ is the component of the flow of $Y_j$ in the $z$-direction, defined as the solution to the ordinary differential equation
	\[ \dot z_j (t) = \eta (z_j (t)) \left( \frac{1}{2} r_j^2 (t) \rho (r_j (t)) + \int_{r_j (t)}^1 s \rho (s) ds \right). \]
In particular, $r_{j, i} (t)$ is constant if the initial condition $z_j (0)$ is sufficiently close to zero.
Moreover,
	\[ \dot r_{j, i} (t) \le \frac{1}{2} \, | \eta' | \left( \int_{0^+}^1 s \rho (s) ds \right)  r_{j, i} (t) = b \cdot r_{j, i} (t), \]
and similarly $- b \cdot r_{j, i} (t) \le \dot r_{j, i} (t)$, and therefore $e^{- b t} r_{j, i} (0) \le r_{j, i} (t) \le e^{b t} r_{j, i} (0)$ for all $0 \le t \le 1$.
Hence
	\[ e^{-b t} \cdot r_j (0) \le r_j (t) \le e^{b t} \cdot r_j (0) \]
for all $j$ and all times $0 \le t \le 1$.
\end{proof}

\begin{lem} \label{lem:z-conv}
The smooth isotopies $\{ \phi_{Y_j Z_j}^t \}$ converge uniformly to a compactly supported continuous isotopy, denoted by $\{ \phi_{Y Z}^t \}$, where $\phi_{Y Z}^t \colon \R^{2 n + 1} \to \R^{2 n + 1}$ is a continuous map for each time $t$.
In fact,
	\[ (\phi_{Y_j}^t)^{-1}_* (- Z_j) = (\phi_{Y_k}^t)^{-1}_* (- Z_k), \ \ \mbox{and} \ \ \{ \phi_{Y Z}^t \} = \{ \phi_{Y_j Z_j}^t \} = \{ \phi_{Y_k Z_k}^t \}, \]
provided $r \ge e^b \epsilon_j$ and $k \ge j$, where $b > 0$ is the same constant as in Lemma~\ref{lem:r-bound}.
\end{lem}

\begin{proof}
Lemma~\ref{lem:r-bound} implies
	\[ e^{-b} \cdot r \le \left( \r \circ (\phi_{Y_j}^t)^{-1} \right) (r_1, \ldots, r_n, \theta_1, \ldots, \theta_n, z) \le e^b \cdot r \]
for all $j$, all $0 \le t \le 1$, and all $(r_1, \ldots, r_n, \theta_1, \ldots, \theta_n, z)$ with $r = \sqrt{r_1^2 + \ldots + r_n^2}$.
Then by Lemma~\ref{lem:y-conv}, $(\phi_{Y_j}^t)^{-1}_* (- Z_j) = (\phi_{Y_k}^t)^{-1}_* (- Z_k)$ for $r \ge e^b \epsilon_j$ and all $k \ge j$.
Since $\r \circ \phi_{Y_j Z_j}^t$ is independent of $t$, it follows that $\phi_{Y_j Z_j}^t = \phi_{Y_k Z_k}^t$ for $r \ge e^b \epsilon_j$ and $k \ge j$, and thus $d (\phi_{Y_j Z_j}^t, \phi_{Y_k Z_k}^t) \le \max \{ 2 r \mid r \le e^b \epsilon_j \} = 2 e^b \epsilon_j \to 0$ as $j \to \infty$.
\end{proof}

\begin{cor} \label{cor:x-conv}
The smooth contact isotopies $\{ \phi_{H_j}^t \} = \{ \phi_{Y_j}^t \circ \phi_{Y_j Z_j}^t \}$ converge uniformly to a continuous isotopy $\{ \phi_t \} = \{ \phi_\rho^t \}$ of maps with compact support, and $\phi_t \colon \R^{2 n + 1} \to \R^{2 n + 1}$ is continuous for each time $t$.
In fact, $\phi_t = \phi_{H_j}^t$ for $r \ge e^b \epsilon_j$.
\end{cor}

\begin{lem}
For each time $t$, the map $\phi_t \colon \R^{2 n + 1} \to \R^{2 n + 1}$ is a homeomorphism.
\end{lem}

\begin{proof}
Since $\phi_t$ is compactly supported, it suffices to prove it is injective as well as surjective.
Continuity of the inverse then follows from a standard argument in point set topology.
By Lemma~\ref{lem:r-bound}, the subset
	\[ A = \left\{ (r_1, \ldots, r_n, \theta_1, \ldots, \theta_n, z) \in \R^{2 n + 1} \mid r > 0 \right\} \subset \R^{2 n + 1} \]
and its complement $B$ in $\R^{2 n + 1}$ are invariant under the isotopy $\{ \phi_t \}$.
It follows again from Lemma~\ref{lem:r-bound} and from the last part of Corollary~\ref{cor:x-conv} that $\phi_t$ is injective and surjective on $A$.
On the other hand, the smooth vector fields
	\[ \eta (z) \left( \int_0^1 s \rho_j (s) ds \right) \frac{\partial}{\partial z} \longrightarrow \eta (z) \left( \int_{0^+}^1 s \rho (s) ds \right) \frac{\partial}{\partial z} \]
uniformly as $j \to \infty$.
The limit is a smooth vector field, and the restriction of the isotopy $\{ \phi_t \}$ to the subspace $B \subset \R^{2 n + 1}$ is determined by the smooth map $\R \to \R$ it generates.
This map is injective and surjective.
\end{proof}

Abbreviate the constants $u (\rho_j, \eta) > 0$ by $u_j$, and again write $u = u (\rho, \eta)$ for the positive number defined by (\ref{eqn:rho-eta}).
Let $U_j = \R^{2 n} \times [- u_j, u_j]$ and $U = \R^{2 n} \times [- u, u]$ denote the corresponding neighborhoods of the origin in $\R^{2 n + 1}$.

\begin{lem} \label{lem:cont-iso-exa}
If $k \ge j$, then $0 < u \le u_j \le u_k$, and therefore $U_j \supset U_k \supset U$ defines a nested sequence of neighborhoods of the origin.
In particular,
\begin{align}
	& \phi_t (r_1, \ldots, r_n, \theta_1, \ldots, \theta_n, z) \label{eqn:iso-near-origin} \\
	& \ = \left (r_1, \ldots, r_n, \theta_1 - t \rho (r), \ldots, \theta_n - t \rho (r), z + t \left( \frac{1}{2} r^2 \rho (r) + \int_r^1 s \rho (s) ds \right) \right) \nonumber
\end{align}
on $U \subset \R^{2 n + 1}$.
When restricted to $U$, the inverse is $(\phi_\rho^t)^{-1} = \phi_{- \rho}^t$.
\end{lem}

\begin{proof}
The lemma follows at once from the hypothesis $0 \le \rho_j \le \rho_k \le \rho$.
\end{proof}

\begin{lem} \label{lem:conformal-factor-exa}
The conformal factors $h_j$ on $[0,1] \times \R^{2 n + 1}$ given by ${(\phi_{H_j}^t)^* \alpha = e^{h_j^t} \alpha}$ converge uniformly to the continuous function $h \colon [0,1] \times \R^{2 n + 1} \to \R$ defined by
	\[ h_t = \int_0^t \left( \eta' (\z) \int_\r^1 v \rho (v) dv \right) \circ \phi_s \, ds, \]
where the second integral is to be interpreted as in equation~(\ref{eqn:finite-int}) at $r = 0$.
The restriction of $h_t$ to $\phi_s^{-1} (U)$ is zero for all $0 \le s \le 1$.
\end{lem}

\begin{proof}
This follows from the identity~(\ref{eqn:conformal-factor-exa}), Corollary~\ref{cor:x-conv}, the construction of the sequence $\rho_j$, and the previous lemma.
\end{proof}

Write $\Phi_H = \{ \phi_t \}$.
By the following corollary, this notation is unambiguous.

\begin{cor}
The triple $(\Phi_H, H, h)$ is a continuous contact dynamical system, and therefore also a topological contact dynamical system.
In particular, its time-one map $\phi$ is a contact homeomorphism.
\end{cor}

\begin{lem} \label{lem:time-one-map-exa}
The time-one map $\phi$ of the continuous isotopy $\{ \phi_t \}$ is not Lipschitz continuous, and in particular, not $C^1$-smooth.
\end{lem}

\begin{proof}
Since $\rho (r) = r^{-a}$ near $r = 0$, there exists a constant $0 < \delta < a$, and two sequences $s_k > s_k' > 0$ that necessarily converge to zero, such that $\rho (s_k) = 0$ and $\rho (s_k') = \pi$ modulo $2 \pi$, and $s_k - s_k' < s_k^{1 + \delta}$.
By identity~(\ref{eqn:iso-exa}) and by Corollary~\ref{cor:x-conv}, near the origin
\begin{align}
	& \phi (r_1, \ldots, r_n, \theta_1, \ldots, \theta_n, z) \label{eqn:time-one-map-exa} \\
	& \ \ \ = \left (r_1, \ldots, r_n, \theta_1 - \rho (r), \ldots, \theta_n - \rho (r), z + \frac{1}{2} r^2 \rho (r) + \int_r^1 s \rho (s) ds \right). \nonumber
\end{align}
Then
	\[ \frac{\left| \phi (s_k, 0, \ldots, 0) - \phi (s_k', 0, \ldots, 0) \right|}{\left| (s_k, 0, \ldots, 0) - (s_k', 0, \ldots, 0) \right|} > \frac{s_k + s_k'}{s_k - s_k'} > \frac{1}{s_k^\delta} \longrightarrow + \infty \]
as $k \to \infty$.
This shows that $\phi$ cannot be Lipschitz continuous.
\end{proof}

Using the above contact homeomorphism $\phi$, one can construct smooth contact vector fields $X_H$ and $X_F$ that are topologically conjugate, but not conjugate by a $C^1$-diffeomorphism that preserves the contact structure.
Cf.\ section 11 in \cite{ms:hvf11}.
Again both $H$ and $F$ are compactly supported inside a Darboux chart, so it suffices to present such examples on standard $\R^{2 n + 1}$.

\begin{exa} \label{exa:conj-vfs}
Let $\phi = \phi_\rho$ be a contact homeomorphism as in Lemma~\ref{lem:time-one-map-exa}, and let $U \subset \R^{2 n + 1}$ denote the neighborhood of the origin defined by (\ref{eqn:rho-eta}).
Let $F \colon \R^{2 n + 1} \to \R$ be an autonomous smooth Hamiltonian that vanishes outside the set $U$, and in a neighborhood of the origin is given by the map
	\[ F (r_1, \ldots, r_n, \theta_1, \ldots, \theta_n, z) = e^{- f (r_1, \theta_1)} \]
if $r_1 > 0$, and zero otherwise, where the smooth map
	\[ f (r, \theta) = \frac{4}{r^2 (1 + 15 \cos^2 \theta)} \]
is the composition of the map $r \mapsto \frac{1}{r^2}$ with the area-preserving change of coordinates $(x, y) \mapsto (2 x, \frac{y}{2})$.
Define $H \colon \R^{2 n + 1} \to \R$ by $H = F \circ \phi = e^{- h} (F \circ \phi)$, since the topological conformal factor $h$ of $\phi$ vanishes on $\phi^{-1} (U)$ by Lemma~\ref{lem:conformal-factor-exa}.
Then by the exponential decay as $r \to 0^+$,
	\[ H (r_1, \ldots, r_n, \theta_1, \ldots, \theta_n, z) = e^{- f (r_1, \theta_1 - \rho (r))} \]
is a smooth map on $\R^{2 n + 1}$, where again $r = \sqrt{r_1^2 + \ldots + r_n^2}$.
Therefore $\Phi_H$ and $\Phi_F$ are smooth contact isotopies, and by Theorem~\ref{thm:trafo-law}, we have $\{ \phi_H^t \} = \{ \phi^{-1} \circ \phi_F^t \circ \phi \}$.
\end{exa}

\begin{lem}
There is no contact $C^1$-diffeomorphism $\psi$ such that the identity $\{ \phi_H^t \} = \{ \psi^{-1} \circ \phi_F^t \circ \psi \}$ holds.
\end{lem}

\begin{proof}
Arguing by contradiction, suppose $\{ \phi_H^t \} = \{ \psi^{-1} \circ \phi_F^t \circ \psi \}$ for a contact $C^1$-diffeomorphism $\psi$.
Then by Lemma~\ref{lem:contact-ham}, $H = e^{- g} (F \circ \psi)$ for the continuous function $g$ on $\R^{2 n + 1}$ that is defined by $\psi^* \alpha = e^g \alpha$, and as a consequence,
\begin{align} \label{eqn:level-sets}
	F = e^{- g \circ \phi^{-1}} \left( F \circ \left( \psi \circ \phi^{-1} \right) \right)
\end{align}
on $\R^{2 n + 1}$.
We may assume that the only isolated zero of the function $F$ on $\R^{2 n + 1}$ is at the origin.
Then $\psi \circ \phi^{-1}$ must fix the origin, and in particular so does the map $\psi$.
Thus (\ref{eqn:level-sets}) implies that in a neighborhood of the origin
	\[ f - f \circ \left( \psi \circ \phi^{-1} \right) = g \circ \phi^{-1}. \]
The level sets $\{ f = c \}$ are of the form $E \times \R^{2 n - 1}$ for concentric ellipses $E$ centered at the origin of the plane that is parameterized by the polar coordinates $(r_1, \theta_1)$.
By continuity of $g$, the restriction of the function $g$ to $\phi^{-1} (U)$ is bounded.
Thus if $r_1 > 0$ is sufficiently small, then the point $\psi \circ \phi^{-1} (r_1, \frac{\pi}{2}, 0)$ lies on the same level set as the point $(\frac{2}{3} r_1, \frac{\pi}{2}, 0)$, or further away from the origin.
Here we write $0$ for the origin in the second factor of the above splitting $\R^{2 n + 1} = \R^2 \times \R^{2 n - 1}$.
Similarly, the point $\psi \circ \phi^{-1} (r_1', \pi, 0)$ lies on the same level set as the point $(\frac{4}{3} r_1', \pi, 0)$, or closer to the origin, provided $r_1' > 0$ is sufficiently small.
For $r_1, r_1' > 0$, the distance of the two concentric ellipses containing the two points $(r_1,\frac{\pi}{2})$ and $(r_1',\pi)$ is $| r_1 - \frac{1}{4} r_1' |$.
Thus if $r_1 > r_1' > 0$ are sufficiently small, then
	\[ \left| \psi \circ \phi^{-1} (r_1, \frac{\pi}{2}, 0) - \psi \circ \phi^{-1} (r_1', \pi, 0) \right| \ge \frac{2}{3} r_1 - \frac{1}{4} \cdot \frac{4}{3} r_1' \ge \frac{1}{3} r_1 > 0. \]

Choose two sequences $s_k > s_k' > 0$ converging to zero, such that $\rho (s_k) = \frac{\pi}{2}$ and $\rho (s_k') = \pi$ modulo $2 \pi$, and such that $s_k - s_k' < s_k^{1 + \delta}$ for a constant $0 < \delta < a$.
Denote by $L$ the Lipschitz constant of the map $\psi$.
Then
	\[ L \ge \frac{\left| \psi \left( \phi^{-1} (s_k, \frac{\pi}{2}, 0) \right) - \psi \left( \phi^{-1} (s_k', \pi, 0) \right) \right|}{\left| \phi^{-1} (s_k, \frac{\pi}{2}, 0) - \phi^{-1} (s_k', \pi, 0) \right|} \ge \frac{\frac{1}{3} s_k}{s_k - s_k'} > \frac{1}{3} \cdot \frac{1}{s_k^\delta} \longrightarrow + \infty \]
as $k \to \infty$.
This contradiction proves the $C^1$-diffeomorphism $\psi$ cannot exist.
\end{proof}

\bibliography{contact}
\bibliographystyle{amsalpha}
\end{document}